\documentclass[final,hidelinks,onefignum,onetabnum]{siamart220329}

\usepackage{enumitem}
\usepackage{subfig}
\usepackage{amsmath}
\usepackage{bm}
\usepackage{tikz}
\usepackage{appendix}
\usepackage{amsfonts}
\usepackage{multirow}
\usepackage{amssymb}
\usepackage{algorithm,algorithmicx,algpseudocode}
\usepackage{extarrows} 

\algdef{SE}[DOWHILE]{Do}{doWhile}{\algorithmicdo}[1]{\algorithmicwhile\ #1}%

\usepackage{amsmath}

\makeatletter 
\g@addto@macro\normalsize{%
	\setlength\abovedisplayskip{6pt plus 2pt minus 2pt}%
	\setlength\belowdisplayskip{6pt plus 2pt minus 2pt}%
	\setlength\abovedisplayshortskip{3pt plus 1pt minus 2pt}%
	\setlength\belowdisplayshortskip{3pt plus 1pt minus 2pt}}
\makeatother

\usepackage{caption}
\usepackage[marginal,hang]{footmisc}

\usepackage{booktabs}

\newsiamremark{remark}{Remark}
\newsiamthm{Def}{Definition}
\newsiamthm{property}{Property}
\newsiamthm{Prop}{Proposition}
\newsiamthm{claim}{Claim}
\newsiamthm{exmp}{Example}
\newsiamthm{assumption}{Assumption}

\ifpdf
\hypersetup{
	pdftitle={},
	pdfauthor={}
}
\fi

\begin{document}

\headers{}{}

\title{{
	Beyond Free-Stream Preservation: Transport Polynomial Exactness for Moving-Mesh Methods under Arbitrary Mesh Motion}
	\thanks{
			This work was partially supported by National Key R\&D Program of China (Grant No.~2022YFA1004500) and National Natural Science Foundation of China (Grant No.~12471390)}
	}
	
\author{
	Chaoyi Cai\thanks{School of Mathematical Sciences, Xiamen University, Xiamen, Fujian 361005, China (\email{caichaoyi@stu.xmu.edu.cn}).}
	\and Qiqin Cheng\thanks{School of Mathematical Sciences, Xiamen University, Xiamen, Fujian 361005, China (\email{qqcheng@stu.xmu.edu.cn}).}
	\and Di Wu\thanks{School of Mathematical Sciences, Xiamen University, Xiamen, Fujian 361005, China (\email{wudiwork@stu.xmu.edu.cn}).}
	\and Jianxian Qiu\thanks{Corresponding author. School of Mathematical Sciences and Fujian Provincial Key Laboratory of Mathematical Modeling and High-Performance Scientific Computing, Xiamen University, Xiamen, Fujian 361005, China (\email{jxqiu@xmu.edu.cn}).}
	}
	
	\maketitle
	
	\begingroup
	\renewcommand\thefootnote{\fnsymbol{footnote}}%
	\setcounter{footnote}{0}%
	\setlength{\footnotemargin}{1.1em}
	\endgroup
\begin{abstract}
High-order moving-mesh methods can effectively reduce numerical diffusion, but their formal accuracy typically relies on the regularity of the mesh velocity.	This dependency creates a fundamental conflict in the numerical solution of hyperbolic conservation laws, where solution-driven adaptation may induce nonsmooth mesh motion, thereby degrading convergence order. To overcome this limitation, we introduce \emph{transport polynomial exactness} (TPE($k$)), a mesh-motion-independent criterion that generalizes classical free-stream preservation (TPE(0)) to the exact advection of degree-$k$ polynomials. We show that the classical geometric conservation law (GCL) is insufficient to ensure TPE($k$) for $k \ge 1$ due to mismatches in higher-order geometric moments. To resolve this, we propose the method of \emph{evolved geometric moments} (EGMs), which evolves geometric information by solving auxiliary transport equations discretized compatibly with the physical variables. We rigorously prove that second-degree EGMs evolved via the third-order strong stability preserving Runge--Kutta (SSPRK3) method coincide with the exact geometric moments. This exactness arises from a \emph{superconvergence} mechanism wherein SSPRK3 reduces to Simpson's rule for EGM evolution. Leveraging this result, we construct a third-order conservative finite-volume rezoning moving-mesh scheme. The scheme satisfies the TPE(2) property for \emph{arbitrary mesh motion} and \emph{any pseudo-time step size}, thereby naturally accommodating spatiotemporally discontinuous mesh velocity. Crucially, this \emph{breaks the efficiency bottleneck} in the conventional advection-based remapping step and reduces the required pseudo-time levels from $\mathcal{O}(h^{-1})$ to $\mathcal{O}(1)$ under bounded but discontinuous mesh velocity. Numerical experiments verify exact quadratic transport and stable third-order convergence under extreme mesh deformation, demonstrating substantial efficiency gains. 

\end{abstract}



\section{Introduction}
Hyperbolic conservation laws are canonical partial differential equations that model local conservation and wave propagation, arising in a wide range of scientific and engineering applications including compressible gas dynamics, plasma physics, multiphase flows, and traffic modeling.
A defining feature of these systems is that, even for smooth initial data, solutions may develop shocks, contact discontinuities, and intricate multiscale structures in finite time. 
Designing conservative numerical methods that are simultaneously high-order accurate in smooth regions and non-oscillatory near discontinuities remains a central challenge in scientific computing \cite{cockburn2012discontinuous,toro2013riemann,shu2020essentially}.

The system of $d$-dimensional hyperbolic conservation laws is given by
\begin{equation}\label{equ:hyperbolic}
	\frac{\partial \mathbf{U}}{\partial t}
	+ \nabla \cdot \mathbf{F}(\mathbf{U}) = \mathbf 0,
\end{equation}
where $\mathbf{U}\in\mathbb{R}^m$ denotes the vector of conserved variables and
$\mathbf{F}(\mathbf{U})\in\mathbb{R}^{m\times d}$ is the flux tensor.
On fixed meshes, high-order conservative schemes such as discontinuous Galerkin (DG) methods \cite{cockburn2012discontinuous} and essentially non-oscillatory (ENO) / weighted ENO (WENO) schemes \cite{shu2020essentially} have reached a high level of maturity.
However, for flows featuring small-scale structures or sharp localized gradients, uniform refinement may be prohibitively expensive, and numerical diffusion may still obscure important features.
This motivates adaptive strategies, including $h$-adaptivity \cite{BERGER198964}, $p$-adaptivity \cite{Mitchell2014}, and $r$-adaptivity based on moving-mesh methods \cite{tang2003adaptive,Budd_Huang_Russell_2009}.

\subsection{Moving-mesh methods: Lagrangian/ALE versus rezoning moving meshes}
For convection-dominated problems with coherent structures, moving-mesh methods can significantly reduce numerical dissipation, improve resolution of discontinuities and interfaces, and enhance the effective accuracy at a fixed number of degrees of freedom \cite{HIRT1974227,tang2003adaptive,Donea2004Chapter1A,Budd_Huang_Russell_2009,anderson2018high}.

Moving-mesh methods can be broadly categorized into two classes. The first class consists of \emph{Lagrangian/arbitrary Lagrangian--Eulerian (ALE)} methods, in which the mesh convects with the flow (or with a prescribed velocity) and the mesh motion is coupled with the numerical flux discretization within each time step. Such methods are widely used for interface tracking and large-deformation multi-material simulations \cite{HIRT1974227,Donea2004Chapter1A}.
The second class is the \emph{rezoning moving-mesh} (RMM) framework. This approach adopts a decoupled strategy within each physical time step: it first advances the solution on the current mesh (physical evolution), then constructs an adapted mesh (rezoning), and finally transfers the numerical solution to the adapted mesh (remapping)~\cite{li2001moving,li2002moving,tang2003adaptive,Budd_Huang_Russell_2009}.
A key advantage of RMM is the decoupling between physics and mesh adaptation:
mature high-order fixed-grid discretizations can be integrated with minimal modification, while the rezoning step accommodates a wide variety of strategies, such as equidistribution principles \cite{huang1994}, variational approaches \cite{BRACKBILL1982342}, spring analogy methods \cite{1990Batina}, and even Lagrangian motion \cite{HIRT1974227}, as reviewed in \cite{Budd_Huang_Russell_2009}. This diversity raises a fundamental question:
given heterogeneous strategies and potential severe mesh deformation, \emph{how can one ensure a moving-mesh method truly delivers its nominal high-order accuracy?}
In the RMM framework, preserving this accuracy critically depends on the remapping step, which serves as the bridge for transferring the high-order solution between evolving meshes.

\subsection{Advection-based remapping and an efficiency bottleneck under discontinuous mesh velocity}\label{subsec:o1/h}
Specifically, remapping can be formulated as a pseudo-time transport process:
at a fixed physical time (e.g., $t=t^{n}$), one introduces a pseudo-time variable $\tau$ and expresses the conservative transfer as
\begin{equation}\label{eq:pseudo_transport}
	\frac{\partial \mathbf U}{\partial \tau}=\mathbf 0,
\end{equation}
which can be discretized in flux form on a moving control volume via the Reynolds transport theorem.
This enables the seamless integration of existing high-order reconstructions and time discretizations \cite{li2001moving,li2002moving,ortega2011geometrically,anderson2015monotonicity,anderson2018high,lipnikov2019high,lipnikov2020conservative,zhangmin2020,gu2023high}, an approach known as \emph{advection-based remapping}.

Despite its appeal, advection-based remapping is, in essence, an ALE discretization of \eqref{eq:pseudo_transport}. Consequently, it inherits the ALE requirement of mesh-velocity regularity to preserve high-order accuracy~\cite{klingenberg2017arbitrary,li2019high,gu2023high}. 
As noted in~\cite{anderson2015monotonicity}, the convergence rate of advection-based remapping ``is observed \emph{only} for `smooth' problems'' where the mesh velocity is also smooth.
However, for hyperbolic conservation laws, solution-driven adaptation often induces spatiotemporally discontinuous mesh velocity, manifesting as steep gradients in its spatial and temporal profiles.
To mitigate the resulting accuracy degradation, two primary strategies are available. The first involves regularizing the mesh motion via smoothing techniques, such as filtering the monitor function or applying geometric smoothing (e.g., Laplacian or harmonic) to the mesh coordinates \cite{tang2003adaptive,anderson2015monotonicity,anderson2018high}. 
While this alleviates the issue of discontinuous mesh velocity, the resulting mesh potentially constrains the spatial resolution of the adapted mesh. The second strategy entails a substantial increase in the number of pseudo-time levels to minimize the incremental displacement. A typical setup prescribes a fixed final pseudo-time (e.g., $\tau^{\mathrm{final}}=1$) and enforces a CFL-like condition $\Delta\tau=\mathcal{O}(h)$, where $h$ is a characteristic mesh size \cite{lipnikov2019high,lipnikov2020conservative}. Although advanced strategies exist to adapt $\tau^{\mathrm{final}}$ based on mesh-velocity smoothness~\cite{gu2023high}, dealing with bounded yet discontinuous mesh velocity still requires $N_\ell=\mathcal{O}(h^{-1})$ pseudo-time levels per physical step to recover nominal high-order accuracy. Without such treatments, a clear loss of convergence order is observed, as reported in~\cite{zhangmin2020}. Since explicit physical time stepping generally uses $\Delta t=\mathcal{O}(h)$, coupling each physical step with $N_\ell=\mathcal{O}(h^{-1})$ remapping levels leads to an overall computational cost equivalent to using an effective time step size of $\mathcal{O}(h^2)$, which can be prohibitively expensive in practice.

Notably, this particular bottleneck is intrinsic to advection-based remapping.
In contrast, intersection-based remapping \cite{dukowicz1984conservative,dukowicz1987accurate} and swept-/flux-based remapping \cite{dukowicz2000incremental} derive the transfer operator directly from geometric overlap/swept regions and therefore do not impose the same mesh-velocity smoothness requirements.
The trade-off is that they rely on complex and robust geometric operations, which become costly and intricate for general polygonal or curvilinear cells.
In this paper, we focus on advection-based remapping and address its efficiency bottleneck by introducing a high-order exactness criterion and a corresponding remapping method. 

\subsection{From free-stream preservation to a stronger high-order criterion: transport polynomial exactness}
A well-known basic requirement in ALE methods is the \emph{geometric conservation law} (GCL), which ensures that the free stream (constant state) is not spuriously perturbed by mesh motion \cite{ThomasLombard1979,farhat2001discrete,mavriplis2006construction,persson2009discontinuous,boscheri2014direct}. 
However, the classical GCL typically serves to enforce compatibility only for the \emph{cell volume} $\int 1\,\mathrm{d}x\mathrm{d}y$ (taking $d=2$ for ease of notation). 
High-order schemes also depend on \emph{higher-order geometric moments} $\int x^s y^r\,\mathrm{d}x\mathrm{d}y$ to maintain accuracy.
In the presence of strongly nonsmooth mesh velocity, satisfying only the GCL may be insufficient to guarantee high-order accuracy or even stable convergence, a phenomenon that will be systematically demonstrated in our numerical experiments. 

To assess high-order accuracy under arbitrary mesh motion and eliminate reliance on mesh-velocity smoothness assumptions (e.g., Lipschitz continuity), we introduce a mesh-motion-independent criterion:
\emph{transport polynomial exactness of order $k$} (TPE($k$)).
Consider the linear advection system 
\begin{equation}\label{eq:advection}
	\frac{\partial \mathbf U}{\partial t}+\mathbf a\cdot\nabla \mathbf U=\mathbf 0,\qquad \mathbf a\in\mathbb{R}^d,
\end{equation}
whose exact solution is given by $\mathbf U(t,\mathbf x)=\mathbf U_0(\mathbf x-\mathbf a t)$ \cite{toro2013riemann}, reflecting the translation invariance of linear transport.
We say that a moving-mesh method for the hyperbolic system \eqref{equ:hyperbolic} is TPE$(k)$ if, whenever \eqref{equ:hyperbolic} reduces to \eqref{eq:advection} and the initial condition $\mathbf U_0$ is a vector-valued polynomial of degree $k$, the fully discrete numerical solution coincides with the exact solution. Under this criterion, TPE$(0)$ reduces to free-stream preservation and is precisely the property enforced by the classical GCL.
Although nonlinear ingredients like limiters may destroy polynomial exactness, requiring TPE($k$) is essential for linear high-order discretizations. Here, $k$ denotes the degree of the underlying polynomial space.

\subsection{Main contributions}
Addressing the challenge of preserving high-order accuracy under discontinuous mesh velocity, we make the following key contributions:

\begin{itemize}
	\item \textbf{Formalization of transport polynomial exactness (TPE).}
	We introduce the TPE$(k)$ property as a verifiable criterion generalizing free-stream preservation (TPE$(0)$) to the exact advection of degree-$k$ polynomials. This provides a sharp diagnostic for assessing whether a moving-mesh method truly delivers its nominal high-order accuracy.
	
	\item \textbf{Generalization of classical GCL via evolved geometric moments (EGMs).} 
	We formulate auxiliary transport equations for geometric moments ($\int x^s y^r\,\mathrm{d}x\,\mathrm{d}y$) and employ a discretization consistent with the physical conserved variables. This formulation extends the classical GCL to higher orders, systematically eliminating geometric mismatch.
	
	\item \textbf{Exact geometric consistency without nonconservative post-processing.}
	We establish the nontrivial result that, under the third-order strong stability preserving Runge--Kutta (SSPRK3) method, the second-degree EGMs coincide with the exact geometric moments, thereby avoiding the need for nonconservative post-processing such as geometric refreshing or source-term corrections \cite{VISBAL2002155,cai2025geometricperturbationrobustcutcellschemetwomaterial}. Remarkably, this alignment emerges as a consequence of superconvergence.
	
	\item \textbf{A TPE(2) remapping operator for arbitrary mesh motion.}
	By coupling a 2-exact hybrid WENO reconstruction with EGMs, we construct an advection-based remapping operator that provably satisfies the TPE(2) property under arbitrary mesh motion. Crucially, this property holds independently of the pseudo-time step size $\Delta\tau$.
	
	\item \textbf{$\mathcal{O}(1)$ pseudo-time levels under discontinuous mesh velocity.}
	Leveraging the $\Delta\tau$-independence of the TPE(2) property, the proposed remapping maintains third-order accuracy using only $\mathcal{O}(1)$ pseudo-time levels, even in the presence of spatiotemporally discontinuous mesh velocity. This yields substantial efficiency gains over conventional $\mathcal{O}(h^{-1})$-level strategies.
	
	\item \textbf{Mapping- and Jacobian-free TPE(2) RMM scheme.} 
	We employ the $2$-exact hybrid WENO reconstruction to realize a TPE(2) fixed-grid physical evolution operator. Coupling this with the TPE(2) remapping operator yields a complete TPE(2) RMM scheme. Notably, the use of EGMs avoids repeated reference-to-physical mappings and Jacobian evaluations, thereby streamlining the implementation.
\end{itemize}

\vspace{0.3em}
\noindent\textbf{Organization of the paper.}
Section~\ref{sec:framework} formulates the TPE$(k)$ property and outlines the structure of the proposed TPE(2) RMM scheme.
Section~\ref{sec:details} presents the algorithmic construction of the physical evolution and remapping operators.
Section~\ref{sec:analysis} provides a rigorous analysis of the TPE(2) property, conservation, and geometric consistency.
Section~\ref{sec:numerical_tests} validates the scheme through carefully designed numerical experiments, highlighting its accuracy, efficiency, and robustness.
Finally, Section~\ref{sec:conclusion} offers concluding remarks.
\section{Transport polynomial exactness and the TPE(2) RMM scheme}\label{sec:framework}
In this section, we formalize the general concept of TPE($k$) and outline the proposed two-dimensional (2D) finite-volume TPE(2) RMM scheme. The detailed algorithmic construction is presented in Section \ref{sec:details}, while the rigorous theoretical analysis is deferred to Section \ref{sec:analysis}.

\subsection{TPE($k$): transport polynomial exactness of order $k$}\label{sec:tpe_definition}
In practical computations, mesh motion can be governed by a diverse array of strategies, including equidistribution principles, variational approaches, spring analogy methods, and Lagrangian motion \cite{HIRT1974227,BRACKBILL1982342,1990Batina,huang1994,Budd_Huang_Russell_2009}. Given this algorithmic heterogeneity and the potential for severe mesh distortions, it is necessary to establish a rigorous, mesh-motion-independent criterion that guarantees genuinely high-order accuracy. 
To generalize the classical free-stream preservation (exactness for constant states) to higher-order accuracy, we consider the linear advection system
\begin{equation}\label{equ:linearAdv}
	\frac{\partial \mathbf U}{\partial t}
	+ \mathbf{a} \cdot \nabla \mathbf U = \mathbf 0
\end{equation}
with initial condition $\mathbf U(t=0,\mathbf{x}) = \mathbf U_0(\mathbf{x})\in \mathbb{R}^m$, where $\mathbf{a} \in \mathbb{R}^d$ represents a constant advection velocity and $\mathbf{x} \in \mathbb{R}^d$ denotes the spatial coordinate. As a prototypical hyperbolic conservation law, equation~\eqref{equ:linearAdv} admits the exact solution
\[
\mathbf U(t,\mathbf{x}) = \mathbf U_0(\mathbf{x} - \mathbf{a} t).
\]
This identity characterizes the fundamental transport mechanism of hyperbolic systems. Ideally, a moving-mesh method should preserve this structural invariance regardless of mesh motion, a requirement that motivates the following definition.

\begin{definition}[Transport polynomial exactness]\label{def:TPE}
	Let $k\in\mathbb N_0$ and define
	\[
	\mathbb{P}_k^m
	= \Bigl\{
	\mathbf p:\mathbb R^d\to\mathbb R^m
	\;\Big|\;
	\mathbf p(\mathbf x)
	= \sum_{|\alpha|\le k} \mathbf c_\alpha \mathbf x^\alpha,
	\ \mathbf c_\alpha\in\mathbb R^m
	\Bigr\},
	\]
	where $\alpha$ is a multi-index and 
	$\mathbf x^\alpha = x_1^{\alpha_1}\cdots x_d^{\alpha_d}$.
	A scheme for system~\eqref{equ:hyperbolic} is said to be \emph{TPE($k$)} if,
	whenever \eqref{equ:hyperbolic} reduces to the linear advection system
	\eqref{equ:linearAdv} with initial condition $\mathbf U_0 \in \mathbb{P}_k^m$,
	the fully discrete solution coincides with the projection of the exact solution
	$
	\mathbf U(t,\mathbf{x}) = \mathbf U_0(\mathbf{x} - \mathbf{a} t)
	$
	onto the discrete space for all discrete times $t$.
\end{definition}

Although nonlinear ingredients (e.g., limiters or nonlinear reconstructions) may compromise this exactness, the linear baseline of a high-order scheme should achieve TPE($k$), consistent with its degree-$k$ polynomial approximation space.

\subsection{Notation}
Achieving TPE($k$) for $k \ge 1$ is nontrivial, as it demands exact compatibility between geometric evolution and physical discretization. In this work, we construct a finite-volume RMM scheme for the hyperbolic conservation laws \eqref{equ:hyperbolic} that rigorously satisfies the TPE(2) property. To describe this scheme, we first introduce the necessary notation. Focusing on the 2D case ($d=2$), we let $\mathbf{x}=(x,y)^\top$ denote the spatial coordinates.

\begin{description}
	\item[Mesh partition $\mathcal{T}_h^n$.] 
	At each time level $t^n$, the computational domain is partitioned into a mesh $\mathcal{T}_h^n = \{ I_{i,j}^n \}_{i,j=1}^{N_x, N_y}$ comprising $N_x \times N_y$ quadrilateral cells. Although the initial mesh $\mathcal{T}_h^0$ is Cartesian (uniform rectangular), the cells $I_{i,j}^n$ evolve into general quadrilaterals during the simulation, governed by the specific rezoning strategy.
	
	\item[Cell averages.] 
	We denote the cell-averaged conserved variables over the control volume $I_{i,j}^n$ by $\overline{\mathbf{U}}_{i,j}^n$.
	
	\item[Physical evolution operator $\mathcal{E}_h$.] 
	This operator advances the solution of \eqref{equ:hyperbolic} on the \emph{fixed} mesh $\mathcal{T}_h^n$ over the physical time step $\Delta t = t^{n+1}-t^n$. It yields the intermediate cell averages $\overline{\mathbf{U}}_{i,j}^{\,n+1,0}$ associated with the pre-rezoned mesh $\mathcal{T}_h^n$.
	
	\item[Rezoning operator $\mathcal{Z}_h$.] 
	Based on the intermediate solution $\overline{\mathbf{U}}_{i,j}^{\,n+1,0}$, the rezoning operator $\mathcal{Z}_h$ generates the updated mesh $\mathcal{T}_h^{n+1}$ at time $t^{n+1}$.
	
	\item[Remapping operator $\mathcal{R}_h$.] 
	The remapping operator $\mathcal{R}_h$ transfers the solution from the current mesh $\mathcal{T}_h^n$ to the updated mesh $\mathcal{T}_h^{n+1}$. This is done by solving the pseudo-time equation $\partial_\tau \mathbf{U} = \mathbf 0$ over the interval $[0, \tau^{\text{final}}]$ using $N_\ell$ pseudo-time levels.
	
	\item[Physical-time vs.\ pseudo-time mesh velocities.] 
	Let $\delta\mathbf{x} := \mathbf{x}^{n+1}-\mathbf{x}^n$ denote the mesh displacement over a physical time step $\Delta t$. We define the \emph{physical-time mesh velocity} as $\mathbf{w}^{\text{phys}} := \delta\mathbf{x}/\Delta t$ and the \emph{pseudo-time mesh velocity} as $\mathbf{w} := \delta\mathbf{x}/\tau^{\text{final}}$. 
	Since the numerical solution is independent of $\tau^{\text{final}}$ (Remark~\ref{remark:tautarget_arbitrary}), we set $\tau^{\text{final}} = \Delta t$ without loss of generality, yielding $\mathbf{w} = \mathbf{w}^{\text{phys}}$. Unless otherwise specified, the term ``mesh velocity'' refers to this unified quantity. We allow $\mathbf{w}$ to be ``discontinuous'' in space and time, referring to steep gradients that violate the Lipschitz regularity typically assumed in error analysis \cite{klingenberg2017arbitrary}.
\end{description}

\subsection{The TPE(2) RMM scheme: overview}
The algorithmic workflow of the TPE(2) RMM scheme for the hyperbolic conservation laws \eqref{equ:hyperbolic} is outlined below.
\begin{center}
	\vspace{0.2cm}
	\fcolorbox{black}{blue!10}{%
		\parbox{.95\linewidth}{%
			\small
			\begin{center}
				\textbf{Overview of the TPE(2) RMM scheme}
			\end{center}
			
			\begin{description}
				\item[\textbf{Step 1 (Initialization).}]
				Initialize the mesh $\mathcal{T}_h^{0}$ and the cell averages
				$\overline{\mathbf{U}}_{i,j}^{0}$ at $t=0$.
				
				\item[\textbf{Step 2 (Physical evolution).}]
				Advance the solution on the fixed mesh $\mathcal{T}_h^{n}$ to $t^{n+1}=t^n+\Delta t$:
				$
				\overline{\mathbf{U}}_h^{\,n+1,0}
				= \mathcal{E}_h\bigl(\overline{\mathbf{U}}_{h}^{n};\,\mathcal{T}_h^{n},\Delta t\bigr).
				$
				
				\item[\textbf{Step 3 (Rezoning).}]
				Generate an adapted mesh at $t^{n+1}$:
				$
				\mathcal{T}_h^{n+1}
				= \mathcal{Z}_h\bigl(\mathcal{T}_h^{n};\,\overline{\mathbf{U}}_h^{\,n+1,0}\bigr).
				$
				
				\item[\textbf{Step 4 (Remapping).}]
				Transfer the solution to the new mesh by solving a
				pseudo-time problem $\partial_\tau \mathbf{U} = \mathbf 0$ on $[0,\tau^{\text{final}}]$:
				$
				\overline{\mathbf{U}}_h^{n+1}
				= \overline{\mathbf{U}}_h^{n+1,N_\ell}
				= \mathcal{R}_h\bigl(
				\overline{\mathbf{U}}_h^{\,n+1,0};
				\,\mathcal{T}_h^n \!\to\! \mathcal{T}_h^{n+1}
				\bigr).
				$
				
				\item[\textbf{Step 5 (Mesh iteration [Optional]).}] If iterative mesh relaxation is required, update $\mathcal{T}_h^n \gets \mathcal{T}_h^{n+1}$ and $\overline{\mathbf{U}}_h^{n+1,0} \gets \overline{\mathbf{U}}_h^{n+1}$, then repeat Steps 3--4.
				
				\item[\textbf{Step 6 (Iteration).}]
				Set $n \gets n+1$ and repeat Steps 2--5 until the final time $t = T$.
			\end{description}
		}%
	}
\vspace{0.3cm}
\end{center}

\section{Algorithmic construction}\label{sec:details}
This section details the algorithmic implementation of the proposed TPE(2) RMM scheme. 
Since $\mathcal{R}_h$ is the key to achieving the TPE(2) property, we prioritize the description of it in Section~\ref{subsec:remap}, followed by the details of $\mathcal{E}_h$ in Section~\ref{sec:phys_evol}. Finally, we describe the rezoning operator $\mathcal{Z}_h$ in Section~\ref{subsec:rezoning} and the 2-exact hybrid WENO reconstruction in Section~\ref{subsec:WENOZQ3}, the latter providing the spatial basis for the TPE(2) property.

\subsection{Remapping operator $\mathcal{R}_h$}\label{subsec:remap}

We perform the remapping by solving the pseudo-time equation $\partial_\tau \mathbf{U} = \mathbf 0$. The major challenge lies in discretizing this equation within the ALE framework to ensure exact preservation of quadratic polynomials.

In the context of the finite volume method, solving $\partial_\tau \mathbf{U} = \mathbf 0$ is equivalent to solving the integral conservation laws over a time-dependent control volume $\Omega(\tau)$:
\begin{equation}\label{equ:ut0_int_form}
	\frac{\mathrm{d}}{\mathrm{d}\tau} \int_{\Omega(\tau)} \mathbf{U}\,\mathrm{d}x\,\mathrm{d}y
	= \int_{\partial \Omega(\tau)}(\mathbf{U}\otimes \mathbf{w})\cdot\mathbf{n}\,\mathrm{d}S,
\end{equation}
where $\mathbf{n} = (n_x,n_y)^\top$ is the outward unit normal to the boundary $\partial \Omega(\tau)$, and $\mathbf{w}$ denotes the velocity of the moving boundary. The cell-integrated conserved quantities 
$
\mathbf{V}_{i,j}(\tau) := \int_{I_{i,j}(\tau)} \mathbf{U}(x,y,\tau)\,\mathrm{d}x\mathrm{d}y
$
are adopted as the evolving variables. Unlike fixed-grid schemes where $\mathbf{V}_{i,j}$ relates trivially to cell averages via constant cell volumes, ALE computations evolve $\mathbf{V}_{i,j}$ directly. This necessitates the reconstruction of $\mathbf{U}$ from these integrated values $\{\mathbf{V}_{i,j}\}$.

\subsubsection{Evolved geometric moments (EGMs)}\label{sec:egms}
Consider a target cell $I_{i^*,j^*}$. Standard high-order reconstruction techniques (e.g., ENO/WENO) seek a vector-valued polynomial
\[
\mathbf{p}^{(I_{i^*,j^*})}(x,y) = \sum_{s,r} \mathbf{c}_{s,r}^{(I_{i^*,j^*})}\,x^{s}y^{r}, \quad \mathbf{c}_{s,r}^{(I_{i^*,j^*})}\in\mathbb{R}^m,
\]
constrained to match the integrated conserved variables over every cell in the stencil $S(I_{i^*,j^*})$:
\begin{equation}\label{eq:WENO_reconstruction_target_cell}
	\sum_{s,r}\mathbf{c}_{s,r}^{(I_{i^*,j^*})} \int_{I_{i,j}} x^{s}y^{r}\,\mathrm{d}x\,\mathrm{d}y
	\;=\;
	\mathbf{V}_{i,j},\quad\forall I_{i,j}\in S(I_{i^*,j^*}).
\end{equation}
Equation~\eqref{eq:WENO_reconstruction_target_cell} highlights the dependence on the \emph{high-order geometric moments}, defined as
\begin{equation}\label{def:geo moment M}
	\mathcal{M}_{s,r,i,j}(\tau) = \int_{I_{i,j}(\tau)} x^{s}y^{r}\,\mathrm{d}x\,\mathrm{d}y.
\end{equation}
In general, a $(k+1)$-th order finite volume scheme necessitates $\mathcal{M}_{s,r,i,j}$ for $s+r \leq k$. Thus, for our third-order TPE(2) RMM scheme, it suffices to focus on $s+r \leq 2$.

\begin{figure}[!thbp]
	\centering
	\captionsetup[subfigure]{labelformat=empty}
	\includegraphics[width=1.0\textwidth]{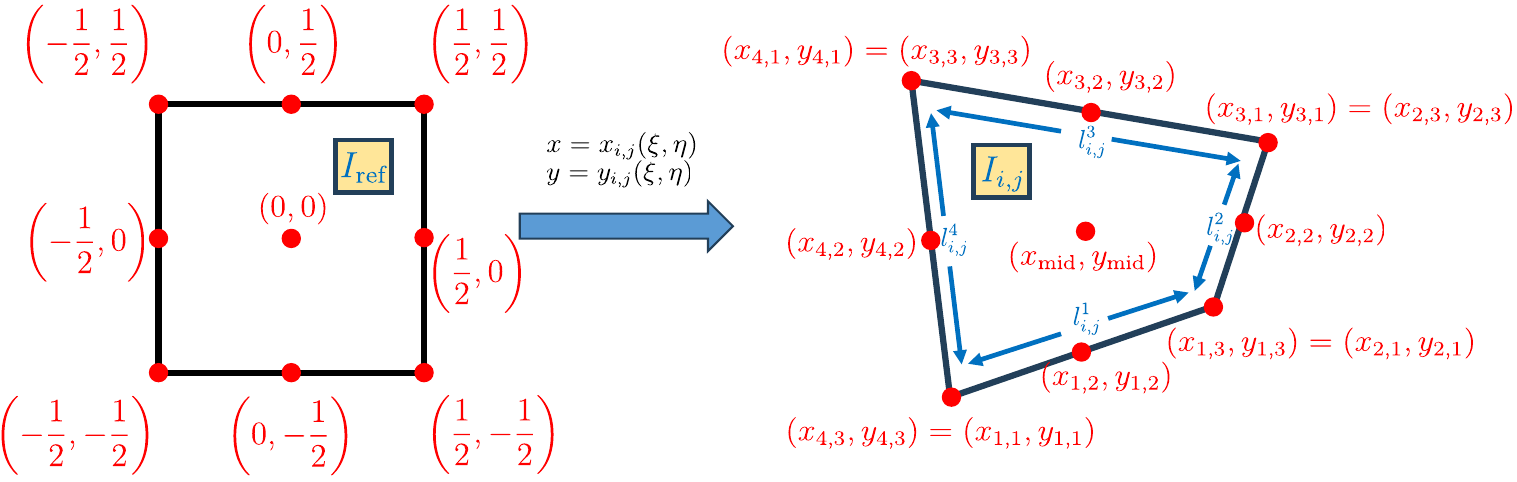}
	\captionsetup{font=small}
	\caption{Mapping from the reference cell $I_{\mathrm{ref}}$ to the physical cell $I_{i,j}$.}
	\label{fig:Jac&Gauss}
\end{figure}
In conventional remapping methods, $\mathcal{M}_{s,r,i,j}$ are typically evaluated via direct integration over the target cell geometry~\cite{gu2023high}.
For the quadrilateral cell $I_{i,j}$ (Figure~\ref{fig:Jac&Gauss}), we employ a bilinear transformation mapping the reference square $I_{\mathrm{ref}}=[-\frac12,\frac12]^2$ to the physical domain $I_{i,j}$:
\begin{equation}\label{equ:shape}
	\begin{aligned}
		x &=x_{i,j}(\xi,\eta)=a_0 + a_1\xi + a_2\eta + a_3\xi\eta,\\
		y &=y_{i,j}(\xi,\eta)=b_0 + b_1\xi + b_2\eta + b_3\xi\eta,
	\end{aligned}
\end{equation}
The geometric moments are then formulated as the integral involving the Jacobian determinant $|J_{i,j}|$:
\begin{equation}\label{equ:shape_int}
	\mathcal{M}_{s,r,i,j}
	= \int_{I_{\mathrm{ref}}} x_{i,j}(\xi,\eta)^s y_{i,j}(\xi,\eta)^r\,|J_{i,j}(\xi,\eta)|\,\mathrm{d}\xi\,\mathrm{d}\eta.
\end{equation}
To evaluate~\eqref{equ:shape_int}, we utilize a $3\!\times\!3$ tensor-product Gauss--Lobatto quadrature rule:
\begin{equation}\label{equ:shape_discre_int}
	\mathcal{M}_{s,r,i,j}
	= \sum_{\ell=1}^3 \sum_{k=1}^3
	\omega_\ell \omega_k\,
	x_{i,j}(\xi_\ell,\eta_k)^{s}\,
	y_{i,j}(\xi_\ell,\eta_k)^{r}\,
	\bigl|J_{i,j}(\xi_\ell,\eta_k)\bigr|.
\end{equation}
Here, the nodes are $\xi_{2\pm1} = \eta_{2\pm1} = \pm \frac{1}{2}$ and $\xi_2 = \eta_2 = 0$, with corresponding weights $\omega_{2\pm1} = \frac{1}{6}$ and $\omega_2 = \frac{2}{3}$. This choice ensures exact integration for degrees $s+r \le 2$.

While the geometric moments $\mathcal{M}_{s,r,i,j}$ can thus be computed exactly, the quantities $\mathbf{V}_{i,j}$ are evolved numerically via~\eqref{equ:ut0_int_form} and inherently contain discretization errors. Coupling ``inexact'' $\mathbf{V}_{i,j}$ with ``exact'' $\mathcal{M}_{s,r,i,j}$ during reconstruction introduces a mismatch that degrades accuracy and can preclude convergence. The classical GCL~\cite{zhangmin2020,lipnikov2020conservative} resolves this only for the zeroth moment (volume), leaving higher-order moments ($s+r>0$) unconstrained. As demonstrated in Section~\ref{sec:numerical_tests}, in regimes with discontinuous mesh velocity, satisfying the classical GCL proves insufficient to guarantee high-order accuracy.

To ensure compatibility with the quantities $\mathbf{V}_{i,j}$, we derive the temporal evolution of the geometric moments. Applying the Reynolds transport theorem to \eqref{def:geo moment M} yields the auxiliary equations
\begin{equation}\label{equ:high order gcl}
	\frac{\mathrm{d}\mathcal{M}_{s,r,i,j}}{\mathrm{d}\tau}
	= \frac{\mathrm{d}}{\mathrm{d}\tau}\int_{I_{i,j}(\tau)} x^{s} y^{r}\,\mathrm{d}x\,\mathrm{d}y
	= \int_{\partial I_{i,j}(\tau)} x^{s} y^{r}\,(\mathbf{w}\cdot\mathbf{n})\,\mathrm{d}S.
\end{equation} 
The boundary-flux structure in~\eqref{equ:high order gcl} mirrors that of the integral conservation laws~\eqref{equ:ut0_int_form}. Discretizing~\eqref{equ:high order gcl} using a numerical scheme compatible with~\eqref{equ:ut0_int_form} yields the \emph{evolved geometric moments} $\widetilde{\mathcal{M}}_{s,r,i,j}$, or simply EGMs. This formulation subsumes the classical GCL within a broader hierarchy: the zeroth moment $\widetilde{\mathcal{M}}_{0,0,i,j}$ exactly recovers the classical GCL, while the higher-order EGMs extend geometric compatibility to higher degrees. As will be demonstrated in Section~\ref{subsubsec:TPE(2)ofRemap}, the incorporation of EGMs endows the remapping operator with the TPE(2) property, which is pivotal for rendering the entire RMM scheme TPE(2).

\subsubsection{Computation of the mesh velocity}
\label{sec:compute_w}
\begin{figure}[!tbhp]
	\centering
	\captionsetup[subfigure]{labelformat=empty}
	\includegraphics[width=0.96\textwidth]{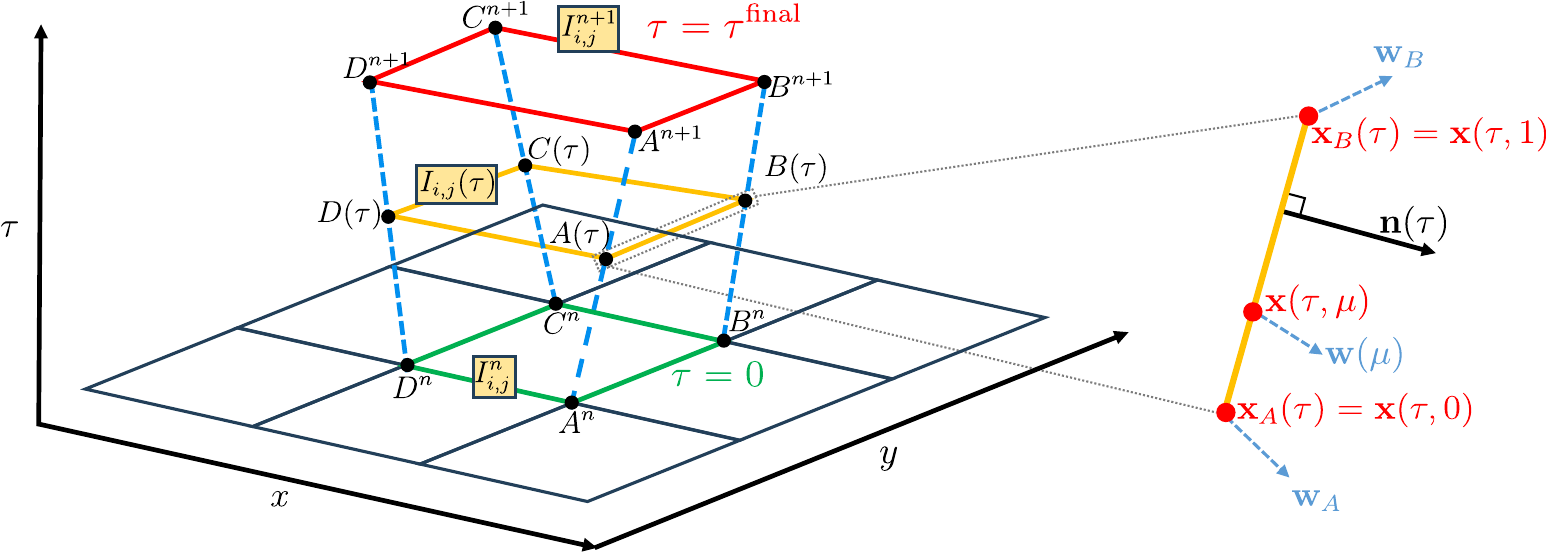}
	\captionsetup{font=small}
	
	
	\caption{Geometric evolution of the cell $I_{i,j}(\tau)=A(\tau)B(\tau)C(\tau)D(\tau)$ during a remapping step. The pseudo-time $\tau$ parametrizes the deformation from the current mesh $\mathcal{T}_h^n$ ($\tau=0$) to the updated mesh $\mathcal{T}_h^{n+1}$ ($\tau=\tau^{\mathrm{final}}$).}
	
	\label{fig:moving_mesh}
\end{figure}

Prior to discretizing the integral equations \eqref{equ:ut0_int_form} and \eqref{equ:high order gcl}, the mesh velocity $\mathbf{w}$ must be determined. As illustrated in Figure~\ref{fig:moving_mesh}, we consider a cell
$
I^n = A^nB^nC^nD^n
$
at time $t^n$.

Let $\mathbf{x}_P^n$ and $\mathbf{x}_P^{n+1}$ ($P \in \{A, B, C, D\}$) denote the vertex positions at $t^n$ and $t^{n+1}$ (determined by the rezoning operator $\mathcal Z_h$). The corresponding vertex velocities are computed as
\begin{equation}\label{equ:grid_vel}
	\mathbf{w}_P = \frac{\mathbf{x}_P^{n+1} - \mathbf{x}_P^n}{\tau^{\mathrm{final}}}, \qquad \text{for } P \in \{A, B, C, D\}.
\end{equation}
These vertex velocities are held constant during the remapping step, while mesh velocities at quadrature points are computed via linear interpolation, as suggested in~\cite{klingenberg2017arbitrary,zhangmin2020,gu2023high}. We simply set $\tau^{\text{final}} = \Delta t$, noting the solution is invariant to the value of $\tau^{\text{final}}$ (Remark~\ref{remark:tautarget_arbitrary}).

\subsubsection{Spatial discretization of $\widetilde{\mathcal{M}}$ and $\mathbf{V}$}\label{sec:remap_spatial}

Consider a generic cell $I_{i,j}$ bounded by edges $\{l_{i,j}^k\}_{k=1}^4$, as illustrated in Figure~\ref{fig:Jac&Gauss}. We spatially discretize \eqref{equ:ut0_int_form} and \eqref{equ:high order gcl} by applying a three-point Gauss--Lobatto quadrature rule along each edge:
\begin{equation}\label{equ:Mevo_semi}
	\frac{\mathrm{d}\widetilde{\mathcal{M}}_{s,r,i,j}}{\mathrm{d}\tau}
	= \sum_{k=1}^4 |l_{i,j}^k| \sum_{g=1}^3
	\omega_g\, x_{k,g}^{\,s} y_{k,g}^{\,r}\,
	(\mathbf{w}_{k,g}\!\cdot\!\mathbf{n}_k)
	\;=:\;
	\mathcal{L}^{\widetilde{\mathcal{M}}_{s,r}}_{i,j}(I_{i,j}),
\end{equation}
\begin{equation}\label{equ:V_semi}
	\frac{\mathrm{d}\mathbf{V}_{i,j}}{\mathrm{d}\tau}
	= -\sum_{k=1}^4 |l_{i,j}^k| \sum_{g=1}^3
	\omega_g\,
	\widehat{\mathbf{G}}\!\left(\mathbf{U}^{\mathrm{int}}_{k,g},
	\mathbf{U}^{\mathrm{ext}}_{k,g},\mathbf{w}_{k,g}\right) 
	\;=:\; \mathcal{L}^{\mathbf{V}}_{i,j}(\widetilde{\mathcal{M}},\mathbf{V},I_{i,j}).
\end{equation}
Here, $|l_{i,j}^k|$ represents the length of edge $l_{i,j}^k$, and $\mathbf{n}_k$ denotes the associated outward unit normal. The terms $\mathbf{x}_{k,g}$, $\mathbf{w}_{k,g}$, and $\omega_g$ correspond to the coordinate, mesh velocity, and weight of the $g$-th quadrature point  on $l_{i,j}^k$, respectively. The numerical flux $\widehat{\mathbf{G}}$ is evaluated using the interior ($\mathbf{U}^{\mathrm{int}}_{k,g}$) and exterior ($\mathbf{U}^{\mathrm{ext}}_{k,g}$) reconstructed point values at $\mathbf{x}_{k,g}$. 
These point values are evaluated from the cell averages $\overline{\mathbf{U}}_{i,j} = \mathbf{V}_{i,j}/\widetilde{\mathcal{M}}_{0,0,i,j}$ and the EGMs $\widetilde{\mathcal{M}}_{s,r,i,j}$ utilizing the WENO reconstruction in Section~\ref{subsec:WENOZQ3}.
We employ the local Lax--Friedrichs (LLF) flux due to its simplicity:
\begin{equation}\label{equ:llf}
	\widehat{\mathbf{G}}\!\left(\mathbf{U}^{\mathrm{int}}_{k,g},
	\mathbf{U}^{\mathrm{ext}}_{k,g},\mathbf{w}_{k,g}\right)
	\!=\! \frac12 \Bigl[
	\mathbf{G}\bigl(\mathbf{U}^{\mathrm{int}}_{k,g},\mathbf{w}_{k,g}\bigr)\cdot\mathbf{n}_k
	+ \mathbf{G}\bigl(\mathbf{U}^{\mathrm{ext}}_{k,g},\mathbf{w}_{k,g}\bigr)\cdot\mathbf{n}_k
	- a^{\max}_{i,j,k}\bigl(\mathbf{U}^{\mathrm{ext}}_{k,g}
	- \mathbf{U}^{\mathrm{int}}_{k,g}\bigr)
	\Bigr],
\end{equation}
where the flux tensor $\mathbf{G}(\mathbf U,\mathbf{w}) = -\mathbf{U}\otimes\mathbf{w} \in \mathbb{R}^{m\times d}$ and $a^{\max}_{i,j,k}=\max_{g}|\mathbf w_{k,g}\cdot\mathbf n_k|$.

To guarantee the TPE(2) property, the reconstruction should be exact for quadratic polynomials. We formalize this requirement by introducing a modified version of $k$-exactness \cite{barth1993recent}, specifically adapted for the EGMs as follows.

\begin{Def}[$k$-exact reconstruction]\label{def:k_exact}
	A reconstruction is termed \emph{$k$-exact} if it satisfies the following property: for any target cell $I_{i^*,j^*}$ on a fixed mesh geometry, if the cell averages $\overline{\mathbf{U}}_{i,j}$ over the reconstruction stencil are generated by a polynomial $\mathbf p(x,y) = \sum_{s+r\le k}\mathbf{c}_{s,r}\,x^s y^r$ via the relation
	\begin{equation}
		\overline{\mathbf{U}}_{i,j}
		= \frac{1}{\widetilde{\mathcal{M}}_{0,0,i,j}}
		\sum_{s+r\le k}\mathbf{c}_{s,r}\,\widetilde{\mathcal{M}}_{s,r,i,j},
	\end{equation}
	then the reconstructed value (denoted as $\mathbf{U}_h$) at any point $(x,y)$ satisfies
	\[
	\mathbf{U}_h(x,y)
	= \sum_{s+r\le k}\mathbf{c}_{s,r}\,x^s y^r.
	\]
\end{Def}

\begin{remark}
	Note that Definition \ref{def:k_exact} implies that the integral of the polynomial is \emph{approximated} using the EGMs, i.e.,
	$$
	\frac{1}{|I_{i,j}|}\!\int_{I_{i,j}}\!\sum_{s+r\le k}\!\mathbf{c}_{s,r}x^s y^r\mathrm{d}x\mathrm{d}y
	\!= \!\frac{1}{{\mathcal{M}}_{0,0,i,j}}\!\sum_{s+r\le k}\!\mathbf{c}_{s,r}\,{\mathcal{M}}_{s,r,i,j}
	\!\approx\! \frac{1}{\widetilde{\mathcal{M}}_{0,0,i,j}}\!\sum_{s+r\le k}\!\mathbf{c}_{s,r}\,\widetilde{\mathcal{M}}_{s,r,i,j}.
	$$
\end{remark}
 
If the cell averages over the entire computational domain are generated by a single quadratic polynomial $\mathbf p(x,y)=\sum_{s+r\le 2} \mathbf{c}_{s,r}\,x^s y^r$ and the reconstruction is $2$-exact, the reconstructed point values coincide across cell interfaces:
\begin{equation}\label{equ:Ucoincide}
	\mathbf{U}^{\mathrm{int}}_{k,g}
	= \mathbf{U}^{\mathrm{ext}}_{k,g}
	= \mathbf p(x_{k,g}^{\,s},y_{k,g}^{\,r}).
\end{equation}
This observation leads to the following lemma.

\begin{lemma}[Linear relation]\label{lem:linear_relation}
	Assume the reconstruction is $2$-exact. If the cell averages $\overline{\mathbf{U}}_{i,j}$ satisfy
	\[
	\overline{\mathbf{U}}_{i,j}
	= \frac{1}{\widetilde{\mathcal{M}}_{0,0,i,j}}
	\sum_{s+r\le 2}\mathbf{c}_{s,r}\,\widetilde{\mathcal{M}}_{s,r,i,j}
	\]
	for all cells $I_{i,j}$, where $\mathbf{c}_{s,r}\in\mathbb{R}^m$ are constant vectors, then the following relation holds for every cell:
	\begin{equation*}
		\mathcal{L}^{\mathbf{V}}_{i,j}(\widetilde{\mathcal{M}},\mathbf{V},I_{i,j})
		= \sum_{s+r\le 2}
		\mathbf{c}_{s,r}\,
		\mathcal{L}^{\widetilde{\mathcal{M}}_{s,r}}_{i,j}(I_{i,j}).
	\end{equation*}
\end{lemma}

\begin{proof}
	Substituting \eqref{equ:Ucoincide} and the flux definition $\mathbf{G} = -\mathbf{U}\otimes\mathbf{w}$ into \eqref{equ:llf} yields
	\begin{align*}
		\mathcal{L}^{\mathbf{V}}_{i,j}(\widetilde{\mathcal{M}},\mathbf{V},I_{i,j})
		&= \sum_{k=1}^4 |l_{i,j}^k| \sum_{g=1}^3
		\omega_g\,
		(\mathbf{U}^{\mathrm{int}}_{k,g}\otimes\mathbf{w}_{k,g})\!\cdot\!\mathbf{n}_k \\
		&= \sum_{s+r\le 2}\mathbf{c}_{s,r}
		\sum_{k=1}^4 |l_{i,j}^k| \sum_{g=1}^3
		\omega_g\,x_{k,g}^{\,s}y_{k,g}^{\,r}\,
		(\mathbf{w}_{k,g}\!\cdot\!\mathbf{n}_k) \\
		&= \sum_{s+r\le 2}\mathbf{c}_{s,r}\,
		\mathcal{L}^{\widetilde{\mathcal{M}}_{s,r}}_{i,j}(I_{i,j}),
	\end{align*}
	which completes the proof.
\end{proof}

This linear relation is crucial for the TPE(2) property established in Theorem~\ref{thm:poly-exact}. 
In the present work, the required 2-exactness of the reconstruction is attained via the hybrid WENO strategy described in Section~\ref{subsec:WENOZQ3}.

\subsubsection{Full remapping operator $\mathcal{R}_h$}
\label{sec:remap_full}
Within the TPE(2) RMM framework, the physical evolution operator $\mathcal{E}_h$ first advances the numerical solution from $t^n$ to $t^{n+1}$ on the mesh $\mathcal{T}_h^n$, yielding the intermediate cell averages $\overline{\mathbf{U}}^{n+1,0}_{i,j}$. 
To initialize the remapping step at pseudo-time $\tau^0=0$, we convert these averages into integrated conserved variables:
$$\mathbf{V}^0_{i,j} := \overline{\mathbf{U}}^{n+1,0}_{i,j} \, \widetilde{\mathcal{M}}^0_{0,0,i,j},$$
where the EGMs $\widetilde{\mathcal{M}}^0_{s,r,i,j}$ are inherited from the previous time step.

To achieve third-order accuracy in pseudo-time, the semi-discrete system \eqref{equ:Mevo_semi}--\eqref{equ:V_semi} is advanced using the SSPRK3 method. One SSPRK3 update from $\tau^\ell$ to $\tau^{\ell+1}=\tau^\ell+\Delta\tau$ is given by
\begin{equation}\label{equ:ssprk3}
	\begin{cases}
		(\cdot)_{i,j}^{(1)}
		= (\cdot)_{i,j}^{\ell}
		+ \Delta \tau\,\mathcal{L}_{i,j}\!\left(\widetilde{\mathcal{M}}^{\ell},\mathbf{V}^{\ell},I_{i,j}^{\ell}\right),\\[0.5em]
		(\cdot)_{i,j}^{(2)}
		= \dfrac{3}{4}(\cdot)_{i,j}^{\ell}
		+ \dfrac{1}{4}\left[ (\cdot)_{i,j}^{(1)}
		+ \Delta \tau\,\mathcal{L}_{i,j}\!\left(\widetilde{\mathcal{M}}^{(1)},\mathbf{V}^{(1)},I_{i,j}^{(1)}\right)\right],\\[0.5em]
		(\cdot)_{i,j}^{\,\ell+1}
		= \dfrac{1}{3}(\cdot)_{i,j}^{\ell}
		+ \dfrac{2}{3}\left[ (\cdot)_{i,j}^{(2)}
		+ \Delta \tau\,\mathcal{L}_{i,j}\!\left(\widetilde{\mathcal{M}}^{(2)},\mathbf{V}^{(2)},I_{i,j}^{(2)}\right)\right],
	\end{cases}
\end{equation}
where the pseudo-time step size $\Delta \tau$ is determined by the CFL condition.
At each stage, the cell averages are recovered via
\begin{equation}
	\overline{\mathbf{U}}_{i,j} = \frac{\mathbf{V}_{i,j}}{\widetilde{\mathcal{M}}_{0,0,i,j}}.
\end{equation}
Repeating this procedure until $\tau=\tau^{\mathrm{final}}$ completes the remapping, yielding
\begin{equation}
	\overline{\mathbf{U}}_h^{n+1}
	= \overline{\mathbf{U}}_h^{n+1,N_\ell}
	= \mathcal{R}_h\left( \overline{\mathbf{U}}_h^{n+1,0}; \,\mathcal{T}_h^n \!\to\! \mathcal{T}_h^{n+1} \right).
\end{equation}


\vspace{0.5em}
\noindent\textbf{Geometry update.}
As illustrated in Figure~\ref{fig:moving_mesh}, the vertex trajectories are linear in pseudo-time. The position of a vertex $P \in \{A, B, C, D\}$ at pseudo-time $\tau$ is explicitly given by
\begin{equation}\label{equ:vertex_location}
	\mathbf{x}_P(\tau) = \mathbf{x}_P^n + \tau\,\mathbf{w}_P \qquad \text{for } P \in \{A, B, C, D\}.
\end{equation}
Let $I_{i,j}(\tau) = A(\tau)B(\tau)C(\tau)D(\tau)$ denote the cell formed by these vertices. The geometries required for the SSPRK3 stages correspond to:
\begin{equation}\label{equ:vertex_location_RK}
	I_{i,j}^{\ell} = I_{i,j}(\tau^\ell), \quad
	I_{i,j}^{(1)} = I_{i,j}(\tau^\ell + \Delta \tau), \quad
	I_{i,j}^{(2)} = I_{i,j}\left(\tau^\ell + \tfrac{1}{2}\Delta \tau\right).
\end{equation}
These configurations determine the Gauss--Lobatto quadrature points on each edge.

\subsection{TPE(2) physical evolution $\mathcal{E}_h$}\label{sec:phys_evol}
This subsection details the TPE(2) physical evolution operator $\mathcal{E}_h$ on fixed meshes, where the geometric moments are time-invariant. Since $\widetilde{\mathcal{M}}_{s,r,i,j}^\ell = \mathcal{M}_{s,r,i,j}(\tau^\ell)$ for $s+r\leq 2$ (Theorem~\ref{thm:coincide}), we omit the tilde notation throughout this subsection.




Using the notation in \eqref{equ:Mevo_semi}--\eqref{equ:V_semi}, the semi-discrete discretization of the hyperbolic system \eqref{equ:hyperbolic} on a fixed grid is given by
\begin{equation}\label{equ:U_semi}
	\frac{\mathrm{d}\overline{\mathbf U}_{i,j}}{\mathrm{d}t}
	= -\frac{1}{\mathcal M_{0,0}}\sum_{k=1}^4 |l_{i,j}^k| \sum_{g=1}^3
	\omega_g\,
	\widehat{\mathbf{F}}\!\left(\mathbf{U}^{\mathrm{int}}_{k,g},
	\mathbf{U}^{\mathrm{ext}}_{k,g}\right)
	\;=:\; \mathcal{L}^{\overline{\mathbf U}}_{i,j}(\overline{\mathbf U}).
\end{equation}
In \eqref{equ:U_semi}, $\mathbf{U}^{\mathrm{int}}_{k,g}$ and $\mathbf{U}^{\mathrm{ext}}_{k,g}$ are also obtained via the 2-exact hybrid WENO reconstruction described in Section~\ref{subsec:WENOZQ3}. We adopt the LLF numerical flux
\begin{equation}\label{equ:llf_fixed}
	\widehat{\mathbf{F}}\!\left(\mathbf{U}^{\mathrm{int}}_{k,g},
	\mathbf{U}^{\mathrm{ext}}_{k,g}\right)
	= \frac12 \Bigl[
	\mathbf{F}\bigl(\mathbf{U}^{\mathrm{int}}_{k,g}\bigr)\cdot\mathbf{n}_k
	+ \mathbf{F}\bigl(\mathbf{U}^{\mathrm{ext}}_{k,g}\bigr)\cdot\mathbf{n}_k
	- \tilde{a}^{\max}_{i,j,k}\bigl(\mathbf{U}^{\mathrm{ext}}_{k,g}
	- \mathbf{U}^{\mathrm{int}}_{k,g}\bigr)
	\Bigr],
\end{equation}
where $\mathbf{F}$ is the flux tensor defined in \eqref{equ:hyperbolic}. The parameter $\tilde{a}^{\max}_{i,j,k}$ represents the local maximum spectral radius of $\partial(\mathbf{F}(\mathbf{U})\cdot\mathbf{n}_k)/\partial\mathbf{U}$ along $l_{i,j}^k$. The fixed-grid evolution $\mathcal{E}_h$ from physical time $t^n$ to
$t^{n+1} = t^n + \Delta t$ is also performed using the third-order SSPRK3 method:
\begin{equation}\label{equ:ssprk3_fixed}
	\begin{cases}
		\overline{\mathbf U}_{i,j}^{(1)} 
		= \overline{\mathbf U}_{i,j}^{n}
		+ \Delta t\,\mathcal{L}^{\overline{\mathbf U}}_{i,j}\bigl(\overline{\mathbf U}^{n}\bigr),\\[0.4em]
		\overline{\mathbf U}_{i,j}^{(2)}
		= \dfrac{3}{4}\overline{\mathbf U}_{i,j}^{n}
		+ \dfrac{1}{4}\Bigl( \overline{\mathbf U}_{i,j}^{(1)}
		+ \Delta t\,\mathcal{L}^{\overline{\mathbf U}}_{i,j}\bigl(\overline{\mathbf U}^{(1)}\bigr)\Bigr),\\[0.6em]
		\overline{\mathbf U}_{i,j}^{\,n+1,0}
		= \dfrac{1}{3}\overline{\mathbf U}_{i,j}^{n}
		+ \dfrac{2}{3}\Bigl( \overline{\mathbf U}_{i,j}^{(2)}
		+ \Delta t\,\mathcal{L}^{\overline{\mathbf U}}_{i,j}\bigl(\overline{\mathbf U}^{(2)}\bigr)\Bigr).
	\end{cases}
\end{equation}
The fixed-grid evolution operator $\mathcal{E}_h$ satisfies the TPE(2) property, as demonstrated in Proposition \ref{prop:TPE_fixed}.

\subsection{Rezoning operator $\mathcal{Z}_h$ and mesh regularity assumptions}
\label{subsec:rezoning}
While the choice of rezoning strategy significantly influences the overall performance of moving-mesh methods~\cite{huang1994,BRACKBILL1982342,HIRT1974227,Budd_Huang_Russell_2009}, the specific design of such algorithms lies beyond the scope of this work. Notably, the TPE(2) RMM scheme preserves the TPE(2) property regardless of the specific rezoning strategy employed. To ensure the geometric well-posedness of the spatial discretization, we impose the following mesh regularity condition.

\begin{assumption}[Mesh regularity]\label{assump:non-twisting}
	As illustrated in Figure~\ref{fig:moving_mesh}, for every pseudo-time $\tau$ and every cell $I_{i,j}(\tau)$, the following conditions hold:
	\begin{enumerate}
		\item[(i)] The four vertices of $I_{i,j}(\tau)$, denoted by $A_{i,j}(\tau)$, $B_{i,j}(\tau)$, $C_{i,j}(\tau)$, and $D_{i,j}(\tau)$, are nondegenerate (pairwise distinct) and strictly ordered counterclockwise.
		\item[(ii)] $I_{i,j}(\tau)$ remains a convex quadrilateral.
	\end{enumerate}
\end{assumption}

\subsection{2-exact hybrid third-order WENO reconstruction}\label{subsec:WENOZQ3}

To achieve high resolution while effectively suppressing spurious oscillations, we employ the third-order WENO reconstruction proposed in \cite{zhu2018new}. 
However, standard WENO procedures may compromise the capability to preserve quadratic polynomials exactly.
To restore 2-exactness, we adopt a hybrid strategy utilizing the KXRCF troubled-cell indicator $\mathcal{I}_{i,j} \in \{0, 1\}$ \cite{li2010hybrid}. The reconstruction polynomial $\mathbf{U}_h(x,y)$ on cell $I_{i,j}$ is defined as:
\begin{equation}
	\mathbf{U}_h(x,y) = 
	\begin{cases} 
		\mathbf{p}^{\text{opt}}(x,y), & \text{if } \mathcal{I}_{i,j} = 0 \quad (\text{smooth region}), \\
		\mathbf{p}^{\text{WENO}}(x,y), & \text{if } \mathcal{I}_{i,j} = 1 \quad (\text{troubled cell}),
	\end{cases}
\end{equation}
where $\mathbf{p}^{\text{opt}}$ is the optimal quadratic polynomial (which is 2-exact by construction) and $\mathbf{p}^{\text{WENO}}$ is the robust WENO reconstruction. If $\mathbf U_0 \in \mathbb P_2^m$, the KXRCF indicator automatically yields $\mathcal{I}_{i,j} = 0$, thereby maintaining the 2-exactness of the reconstruction.

Theorem~\ref{thm:TPE2-RMM} in Section~\ref{sec:tpe_full} rigorously establishes that, with this 2-exact reconstruction, the proposed RMM scheme fully satisfies the TPE(2) property. This is further verified by numerical experiments in Section~\ref{sec:numerical_tests}.

\begin{remark}[Implementation via EGMs]\label{rem:EGM_reconstruction}
	Consider a target cell $I_{i^*\!,j^*\!}$. To ensure well-conditioned reconstruction, we employ the normalized basis centered at the cell centroid $(x_c, y_c)\!=\!\left(\frac{\widetilde{\mathcal{M}}_{1,0,i^*\!,j^*\!}}{\widetilde{\mathcal{M}}_{0,0,i^*\!,j^*\!}},\frac{\widetilde{\mathcal{M}}_{0,1,i^*\!,j^*\!}}{\widetilde{\mathcal{M}}_{0,0,i^*\!,j^*\!}}\right)$ with scaling length $\hat{h} = \sqrt{\widetilde{\mathcal{M}}_{0,0,i^*\!,j^*\!}}$:
	\begin{equation}\label{eq:norm_basis}
		\psi_{s,r}(x,y) = \left(\frac{x-x_c}{\hat{h}}\right)^s \left(\frac{y-y_c}{\hat{h}}\right)^r, \quad 0 \le s+r \le 2.
	\end{equation}
	The coefficients $\mathbf{c}_{s,r}$ of $\mathbf{p}^{\text{opt}}(x,y) = \sum_{s+r\leq2} \mathbf{c}_{s,r} \psi_{s,r}(x,y)$ are determined by a least-squares fit to the integrated averages on all stencil cells $I_{i,j}\in S_{i^*\!,j^*\!}$. The integrals of the basis functions are evaluated using the EGMs via
	\begin{equation} 
		\int_{I_{i,j}} \psi_{s,r}(x,y) \,\mathrm{d}x\mathrm{d}y = \frac{1}{{\hat{h}}^{s+r}} \sum_{p=0}^s \sum_{q=0}^r \binom{s}{p}\binom{r}{q} (-x_c)^{s-p} (-y_c)^{r-q} \widetilde{\mathcal{M}}_{p,q,i,j}.
	\end{equation}
	This demonstrates that the reconstruction depends solely on $\widetilde{\mathcal{M}}_{s,r}$ with $s+r \le 2$.
\end{remark}

\begin{remark}
	To guarantee the affine invariance of the WENO reconstruction \cite{WANG2022630}, we divide the smoothness indicators $\beta$ in \cite{zhu2018new} by the local variance of the reconstructed variable.
\end{remark}


\section{Theoretical analysis}\label{sec:analysis}
This section provides a rigorous analysis of the proposed TPE(2) RMM scheme. We establish the TPE(2) property for the physical evolution operator \(\mathcal{E}_h\), the remapping operator \(\mathcal{R}_h\), and thus the full RMM scheme. We also prove the geometric consistency of the EGMs, and derive the estimation for the required pseudo-time levels $N_\ell$.

\subsection{The TPE(2) property of the physical evolution operator \(\mathcal{E}_h\)}
\label{subsubsec:TPE(2)offixed} 
Consider the 2D linear advection system 
\begin{equation}\label{equ:linearAdv_fix2d}
	\frac{\partial \mathbf U}{\partial t}
	+ \nabla\cdot\mathbf F(\mathbf U)=\mathbf 0,\qquad \mathbf F(\mathbf U)=\mathbf U\otimes\mathbf{a},\qquad \mathbf a=(a_x,a_y).
\end{equation}
If $\mathbf U_0\in\mathbb{P}_2^m$, the semi-discrete scheme \eqref{equ:U_semi} is spatially exact. This follows because the three-point Gauss--Lobatto rule integrates cubic polynomials exactly \cite[Section 2.2.3]{canuto2006spectral}. Consequently, the SSPRK3 method updates the solution via the operator
\begin{equation}
	\overline{\mathbf{U}}^{n+1} = R(\Delta t \mathcal{L}) \overline{\mathbf{U}}^{n}, \quad R(z) = 1 + z + \frac{z^2}{2} + \frac{z^3}{6},
\end{equation}
where $\mathcal{L}$ represents the exact spatial operator.
Since the exact evolution is $e^{\Delta t \mathcal{L}}$ and $(\mathcal{L})^3 \overline{\mathbf{U}}^{n} = \mathbf{0}$ (as third derivatives of quadratics vanish), the numerical update $R(\Delta t \mathcal{L})$ coincides with the exact evolution $e^{\Delta t \mathcal{L}}$. Thus, we have:

\begin{proposition}[TPE(2) property of $\mathcal{E}_h$]\label{prop:TPE_fixed}
	Consider the linear advection system \eqref{equ:linearAdv_fix2d}. Let the system be solved using the semi-discrete scheme \eqref{equ:U_semi} and the SSPRK3 method \eqref{equ:ssprk3_fixed}. If at time $t^n$,
	\[
	\overline{\mathbf U}_{i,j}^n
	= \frac{1}{|I_{i,j}|} \int_{I_{i,j}} \mathbf{p}(x,y) \,\mathrm{d}x\mathrm{d}y
	= \frac{1}{\mathcal M_{0,0,i,j}} \sum_{s+r\leq 2} \mathbf c_{s,r} \mathcal M_{s,r,i,j}.
	\]
	where $\mathbf{p}(x,y)=\sum_{s+r\leq 2}\mathbf c_{s,r} x^s y^r \in \mathbb{P}_2^m$.
	Then, at time $t^{n+1}=t^n+\Delta t$, 
	\[
	\overline{\mathbf U}_{i,j}^{n+1}
	= \frac{1}{\mathcal M_{0,0,i,j}} \sum_{s+r\leq 2} \mathbf c_{s,r} \left[ \sum_{k=0}^s \sum_{\ell=0}^r \binom{s}{k}\binom{r}{\ell} (-a_x\Delta t)^{s-k} (-a_y\Delta t)^{r-\ell} \mathcal M_{k,\ell,i,j} \right].
	\]
	which coincides with the cell averages of the exact solution $\mathbf{p}(x\!-\!a_x\Delta t,y\!-\!a_y\Delta t)$.
\end{proposition}

\begin{remark}
	The conclusion of Proposition~\ref{prop:TPE_fixed} naturally extends to TPE(3) if the semi-discrete scheme~\eqref{equ:U_semi} employs a 3-exact reconstruction.
\end{remark}

\subsection{The TPE(2) property of the remapping operator \(\mathcal{R}_h\)}
\label{subsubsec:TPE(2)ofRemap}
By utilizing EGMs, the semi-discrete scheme \eqref{equ:Mevo_semi}--\eqref{equ:V_semi} preserves quadratic polynomials exactly.

\begin{theorem}[Quadratic polynomial preservation under Forward Euler]
	\label{thm:poly-exact}
	Let the semi-discrete systems \eqref{equ:Mevo_semi}--\eqref{equ:V_semi} be advanced by the forward Euler method:
	\begin{equation}\label{equ:euler_forward}
		(\cdot)_{i,j}^{\ell+1} 
		= (\cdot)_{i,j}^{\ell} 
		+ \Delta \tau\,\mathcal{L}_{i,j}\bigl(\widetilde{\mathcal{M}}^{\ell},\mathbf{V}^{\ell},I_{i,j}^{\ell}\bigr).
	\end{equation}
	Suppose that at pseudo-time $\tau^\ell$, the integrated variables and geometric moments satisfy the linear relation
	\begin{equation}\label{equ:taun_poly}
		\mathbf{V}_{i,j}^{\ell} = \sum_{s+r\le 2}\mathbf{c}_{s,r}\,\widetilde{\mathcal{M}}_{s,r,i,j}^{\ell}, \quad \forall i,j,
	\end{equation}
	which corresponds to the integration of a quadratic polynomial $\mathbf{p}(x,y) \!=\! \sum_{s+r\le 2}\!\mathbf{c}_{s,r}\,x^s y^r.$
	Then, the coefficients $\{\mathbf{c}_{s,r}\}$ are preserved at the next level $\tau^{\ell+1}$:
	\begin{equation}\label{equ:taunp1_poly_linear}
		\mathbf{V}_{i,j}^{\ell+1} = \sum_{s+r\le 2}\mathbf{c}_{s,r}\,\widetilde{\mathcal{M}}_{s,r,i,j}^{\ell+1}.
	\end{equation} 
	Consequently, the updated cell averages $\mathbf{V}^{\ell+1}/\widetilde{\mathcal{M}}^{\ell+1}_{0,0}$ coincide with the cell averages of the exact solution $\mathbf{p}(x,y)$ in the sense of EGM-based integration. Any 2-exact reconstruction based on $\{\widetilde{\mathcal{M}}^{\ell+1}_{s,r}\!, \mathbf{V}^{\ell+1}\!/\!\widetilde{\mathcal{M}}^{\ell+1}_{0,0}\}$ exactly recovers $\mathbf{p}(x,y)$.
\end{theorem}

\begin{proof}
	By substituting \eqref{equ:taun_poly} into the forward Euler update \eqref{equ:euler_forward} and invoking the linearity property from Lemma~\ref{lem:linear_relation}, we deduce
	\begin{equation*}
		\begin{aligned}
			\mathbf V_{i,j}^{\ell+1}
			&= \mathbf V_{i,j}^{\ell}
			+ \Delta \tau\,\mathcal{L}^{\mathbf{V}}_{i,j}
			\bigl(\widetilde{\mathcal{M}}^\ell,\mathbf{V}^\ell,I_{i,j}^{\ell}\bigr) \\
			&= \sum_{s+r\le 2}\mathbf{c}_{s,r}\,\widetilde{\mathcal{M}}^\ell_{s,r,i,j}
			+ \Delta \tau\sum_{s+r\le 2} \mathbf{c}_{s,r}\,
			\mathcal{L}^{\widetilde{\mathcal{M}}_{s,r}}_{i,j}
			\bigl(I_{i,j}^{\ell}\bigr) \\
			&= \sum_{s+r\leq2}\mathbf{c}_{s,r}
			\left[ \widetilde{\mathcal{M}}_{s,r,i,j}^{\ell}
			+\Delta \tau\,\mathcal{L}^{\widetilde{\mathcal{M}}_{s,r}}_{i,j}
			\bigl(I_{i,j}^{\ell}\bigr) \right] \\
			&= \sum_{s+r\leq2}\mathbf{c}_{s,r}\,\widetilde{\mathcal{M}}_{s,r,i,j}^{\ell+1},
		\end{aligned}
	\end{equation*}
	which completes the proof.
\end{proof}

As the SSPRK3 method \eqref{equ:ssprk3} is formally a convex combination of forward Euler steps, the results of Theorem~\ref{thm:poly-exact} extend immediately to the fully discrete remapping operator $\mathcal{R}_h$. Since remapping solves the transport equation $\partial_\tau \mathbf{U} = \mathbf 0$, we also term this quadratic polynomial preservation as the TPE(2) property.

\begin{proposition}[TPE(2) property of $\mathcal{R}_h$]
	\label{prop:ssprk3_poly_exact}
	The conclusions of Theorem~\ref{thm:poly-exact} remain valid when the semi-discrete system \eqref{equ:Mevo_semi}--\eqref{equ:V_semi} is advanced using the SSPRK3 method.
\end{proposition}

Theorem~\ref{thm:poly-exact} and Proposition~\ref{prop:ssprk3_poly_exact} hold \emph{regardless of the prescribed mesh motion, mesh quality, or the pseudo-time step size $\Delta\tau$}. Consequently, the remapping operator $\mathcal{R}_h$ maintains stability and high-order accuracy without requiring excessive pseudo-time levels, even when the mesh velocity $\mathbf{w}$ is discontinuous. 

\begin{remark}
	The findings of Theorem~\ref{thm:poly-exact} and Proposition~\ref{prop:ssprk3_poly_exact} generalize naturally to polynomials of arbitrary degree $k$. Specifically, $\mathcal{R}_h$ achieves the TPE($k$) property provided that a $k$-exact reconstruction and any SSP time discretization are employed.
\end{remark}

\subsection{The TPE(2) property of the full RMM scheme}\label{sec:tpe_full}

By coupling EGMs with a 2-exact hybrid WENO reconstruction, the full RMM scheme achieves the TPE(2) property.

\begin{theorem}[TPE(2) property of RMM]\label{thm:TPE2-RMM}
	The proposed RMM scheme satisfies the TPE(2) property for any $\Delta t$ and $\Delta \tau$.
\end{theorem}

\begin{proof}
	Consider the linear advection system \eqref{equ:linearAdv} with initial condition $\mathbf{U}_0 \in \mathbb{P}_2^{m}$. Since the reconstruction is 2-exact, the physical evolution operator $\mathcal{E}_h$ is exact on the fixed grid for any $\Delta t$ (by Proposition \ref{prop:TPE_fixed}), and the remapping operator $\mathcal{R}_h$ preserves quadratic polynomials for any $\Delta \tau$ (by Proposition \ref{prop:ssprk3_poly_exact}). Consequently, the composite update over a full time step introduces no discretization error for any $\mathbf{U}_0 \in \mathbb{P}_2^{m}$, thereby satisfying the TPE(2) definition regardless of $\Delta t$ and $\Delta \tau$.
\end{proof}

\subsection{Geometric consistency}
\label{subsec:remapping conservative}
During the pseudo-time advancement, discrepancies may accumulate between the evolved moments $\widetilde{\mathcal{M}}_{s,r,i,j}^{\ell+1}$ and the exact geometric moments $\mathcal M_{s,r,i,j}(\tau^{\ell+1})$. Such inconsistencies can prevent the reconstruction from faithfully matching the actual cell geometry, potentially leading to numerical instability and loss of accuracy. Although methods like geometric refreshing or source-term corrections \cite{VISBAL2002155,cai2025geometricperturbationrobustcutcellschemetwomaterial} can mitigate this, they compromise conservation. Fortunately, our discretization strictly guarantees $
\widetilde{\mathcal M}^\mathrm{\ell+1}_{s,r,i,j}
= \mathcal M_{s,r,i,j}(\tau^{\ell+1})$ for $s+r\leq2$,
thereby obviating the need for any post-processing. This exact identity is a remarkable consequence attributed to the superconvergence of the SSPRK3 method, as detailed in the following theorem. 



\begin{theorem}[Geometric consistency]\label{thm:coincide}
	Under the SSPRK3 method \eqref{equ:ssprk3}, the second-degree EGMs coincide with the exact geometric moments, i.e.,
	\begin{equation*}
		\widetilde{\mathcal{M}}^{\ell+1}_{s,r,i,j}
		= \mathcal M_{s,r,i,j}(\tau^{\ell+1}),
		\qquad \text{for } s+r\leq2.
	\end{equation*}
\end{theorem}
The proof of this theorem is nontrivial and relies on three auxiliary lemmas. We first establish these lemmas in the following subsection, and then present the complete proof of the theorem in Section \ref{subsec:proof_consistency}.

\subsubsection{Auxiliary lemmas}\label{subsec:aux_lems}
The following lemma establishes the intrinsic polynomial structure of the exact geometric moments.
\begin{lemma}\label{lem:M_poly}
	For each moving cell $I_{i,j}(\tau)$, the exact geometric moment $\mathcal M_{s,r,i,j}(\tau)$ is a polynomial in $\tau$ of degree $s+r+2$.
\end{lemma}

\begin{proof}
	Analyzing the temporal dependence of $\mathcal M_{s,r,i,j}(\tau)$ via direct integration is cumbersome. Instead, we differentiate it in pseudo-time and apply
	\eqref{equ:high order gcl} to convert the volume contribution into boundary integrals:
	\[
	\frac{\mathrm{d}\mathcal{M}_{s,r,i,j}}{\mathrm{d}\tau}
	= \int_{\partial I_{i,j}(\tau)} x^{s} y^{r}\,(\mathbf{w}\!\cdot\!\mathbf{n})\,\mathrm{d}S
	= \sum_{k=1}^4\int_{l_{i,j}^k(\tau)} x^{s} y^{r}\,(\mathbf{w}\!\cdot\!\mathbf{n})\,\mathrm{d}S,
	\]
	where $l_{i,j}^k(\tau)$ denotes the $k$th edge of the quadrilateral cell $I_{i,j}(\tau)$.
	
	Consider the edge $l_{i,j}^1(\tau)$ with endpoints $A(\tau)$ and $B(\tau)$, as illustrated in Figure~\ref{fig:moving_mesh} (right). Parameterize points on the segment
	$A(\tau)B(\tau)$ by
	\begin{equation}\label{equ:Mt_poly1}
		\mathbf x(\tau,\mu)
		\!=\! (1-\mu)\,\mathbf x_A(\tau)+\mu\,\mathbf x_B(\tau)
		\!=\! (1-\mu)\bigl(\mathbf x_A(0)+\mathbf w_A\tau\bigr)
		+ \mu\bigl(\mathbf x_B(0)+\mathbf w_B\tau\bigr),
		~~~ \mu\in[0,1].
	\end{equation}
	Differentiation with respect to $\tau$ yields the mesh velocity:
	\begin{equation}\label{equ:Mt_poly2}
		\mathbf w(\mu)
		= \frac{\partial\mathbf x(\tau,\mu)}{\partial \tau}
		= (1-\mu)\,\mathbf w_A+\mu\,\mathbf w_B,
		\qquad \mu\in[0,1],
	\end{equation}
	which is independent of $\tau$. Using the parameterization \eqref{equ:Mt_poly1}, the differential arc length transforms as $\mathrm{d}S = |\mathbf{x}_B(\tau) - \mathbf{x}_A(\tau)|\,\mathrm{d}\mu$. To simplify the derivation, we introduce the \textit{scaled normal} $\mathbf{n}^*(\tau)$, defined directly by rotating the edge vector:
	\begin{equation}\label{equ:Mt_poly4}
		\mathbf{n}^*(\tau)
		:= \mathbf{R}_{-\frac{\pi}{2}} \bigl(\mathbf{x}_B(\tau)-\mathbf{x}_A(\tau)\bigr)
		= \mathbf{R}_{-\frac{\pi}{2}} \bigl[\tau(\mathbf{w}_B-\mathbf{w}_A) + \bigl(\mathbf{x}_B(0)-\mathbf{x}_A(0)\bigr)\bigr],
	\end{equation}
	where $\mathbf{R}_{-\frac{\pi}{2}}\!=\!\bigl( \begin{smallmatrix} 0 & 1 \\ -1 & 0 \end{smallmatrix} \bigr)$.
	Noting that $\mathbf{n}(\tau)\,\mathrm{d}S \equiv \mathbf{n}^*(\tau)\,\mathrm{d}\mu$, the edge integral simplifies to
	\begin{equation}\label{equ:Mt_poly5}
		\int_{l_{i,j}^1(\tau)} x^{s} y^{r}\,(\mathbf{w}\cdot\mathbf{n})\,\mathrm{d}S
		= \int_0^1 [x(\tau,\mu)]^{s} [y(\tau,\mu)]^{r}\,
		\bigl(\mathbf{w}(\mu)\cdot\mathbf{n}^*(\tau)\bigr)\,\mathrm{d}\mu,
	\end{equation}
	where we have explicitly substituted the parameterized coordinates $\mathbf x(\tau,\mu)$ into the monomial basis $x^sy^r$.
	
	Since $\mathbf{x}(\tau,\mu)$ and $\mathbf{w}(\mu)\cdot\mathbf{n}^*{(\tau)}$ are affine in $\tau$ (see \eqref{equ:Mt_poly1}--\eqref{equ:Mt_poly4}), the integrand $[x(\tau,\mu)]^{s} [y(\tau,\mu)]^{r}\,
	\bigl(\mathbf{w}(\mu)\cdot\mathbf{n}^*(\tau)\bigr)$ is a polynomial in $\tau$ of degree at most $s+r+1$. This degree is preserved after integration over $\mu$. Extending this to all edges, $\frac{\mathrm{d}\mathcal{M}{s,r,i,j}}{\mathrm{d}\tau}$ also has degree $s+r+1$, which implies that
	$\mathcal M_{s,r,i,j}(\tau)$ is a polynomial of degree at most $s+r+2$.
\end{proof}

The following lemma establishes the exactness of the semi-discrete operator $\mathcal L^{\widetilde{\mathcal{M}}_{s,r}}_{i,j}$.

\begin{lemma}\label{lem:LM_exact}
	For $s+r\leq2$, the semi-discrete scheme \eqref{equ:Mevo_semi} satisfies
	\begin{equation}\label{equ:LM_exact}
		\mathcal L^{\widetilde{\mathcal M}_{s,r}}_{i,j}
		\bigl(I_{i,j}(\tau)\bigr)
		= \frac{\mathrm{d}\mathcal M_{s,r,i,j}(\tau)}{\mathrm{d}\tau},
	\end{equation}
	where $\mathcal M_{s,r,i,j}(\tau)$ is the exact geometric moment of the moving cell $I_{i,j}(\tau)$.
\end{lemma}

\begin{proof}
	From \eqref{equ:Mevo_semi} we have
	\[
	\mathcal L^{\widetilde{\mathcal{M}}_{s,r}}_{i,j}
	\bigl(I_{i,j}(\tau)\bigr)
	= \sum_{k=1}^4 |l_{i,j}^k(\tau)| \sum_{g=1}^3
	\omega_g\, x_{k,g}^{\,s} y_{k,g}^{\,r}\,
	(\mathbf{w}_{k,g}\!\cdot\!\mathbf{n}_k).
	\]
	To establish \eqref{equ:LM_exact}, it suffices to verify that the three-point Gauss--Lobatto quadrature rule is exact for each edge $k=1,\dots,4$. That is, we need to show
	\begin{equation}\label{equ:threeGL_exact}
		|l_{i,j}^k(\tau)| \sum_{g=1}^3
		\omega_g\, x_{k,g}^{\,s} y_{k,g}^{\,r}\,
		(\mathbf{w}_{k,g}\!\cdot\!\mathbf{n}_k)=\int_{l_{i,j}^k(\tau)} x^{s} y^{r}\,(\mathbf{w}\!\cdot\!\mathbf{n})\,\mathrm{d}S.
	\end{equation}
	Consider the first edge $l_{i,j}^1(\tau)$. We also adopt the parameterization strategy from Lemma~\ref{lem:M_poly}. Since $\mathbf x(\tau,\mu)$ and $\mathbf{w}\cdot\mathbf n^*$ are affine in $\mu$ by \eqref{equ:Mt_poly1}--\eqref{equ:Mt_poly4}, the integrand $[x(\tau,\mu)]^{s} [y(\tau,\mu)]^{r}\,
	\bigl(\mathbf{w}(\mu)\cdot\mathbf{n}^*(\tau)\bigr)$ is a polynomial in $\mu$ of degree $s+r+1$. For $s+r\leq 2$, this degree is at most three. Since the three-point Gauss--Lobatto quadrature rule integrates cubic polynomials exactly \cite[Section 2.2.3]{canuto2006spectral}, equation \eqref{equ:threeGL_exact} holds for $k=1$. By symmetry, the same argument applies to the remaining edges. Summing the contributions from all four edges yields \eqref{equ:LM_exact}, completing the proof.
\end{proof}

The following lemma establishes that for pure time-dependent problems, the SSPRK3 method reduces to Simpson's rule.
\begin{lemma}[Reduction to Simpson's rule]\label{lem:SSPRK3_simpson}
	Consider the purely time-dependent problem $u_\tau = f(\tau)$. The SSPRK3 method \eqref{equ:ssprk3}, with stage evaluations corresponding to $\tau^n$, $\tau^n + \Delta \tau$, and $\tau^n + \frac{\Delta \tau}{2}$, yields the update
	\begin{equation*}
		u^{n+1}-u^n
		= \frac{\Delta \tau}{6}\left[
		f(\tau^n)
		+ 4 f\!\left(\tau^n + \frac{\Delta \tau}{2}\right)
		+ f(\tau^n + \Delta \tau)
		\right].
	\end{equation*}
	This formulation coincides exactly with Simpson's rule over the interval $[\tau^n,\tau^{n+1}]$.
\end{lemma}

\begin{proof}
	For the scalar ODE $u_\tau = f(\tau)$, the spatial operator $\mathcal L$ reduces to the source term $f(\tau)$. Applying the SSPRK3 method \eqref{equ:ssprk3} by evaluating this term at $\tau^n$, $\tau^n + \Delta \tau$, and $\tau^n + \frac{\Delta \tau}{2}$ for the three respective stages leads to:
	\begin{align*}
		u^{(1)}
		&= u^n + \Delta \tau\, f(\tau^n), \\
		u^{(2)}
		&= \frac{3}{4}u^n
		+ \frac{1}{4}\Bigl(u^{(1)} + \Delta \tau\, f(\tau^n+\Delta \tau)\Bigr), \\
		u^{n+1}
		&= \frac{1}{3}u^n
		+ \frac{2}{3}\Bigl(u^{(2)} + \Delta \tau\, f\bigl(\tau^n+\tfrac{\Delta \tau}{2}\bigr)\Bigr). 
	\end{align*}
	Simple algebraic rearrangement gives
	\begin{equation}
		u^{n+1} - u^n= \frac{\Delta \tau}{6}\left[
		f(\tau^n)
		+ 4 f\!\left(\tau^n + \tfrac{\Delta \tau}{2}\right)
		+ f(\tau^n + \Delta \tau)
		\right].
	\end{equation}
	This proves the claim.
\end{proof}

Lemma~\ref{lem:M_poly} establishes that $\mathcal{M}_{s,r,i,j}$ are quartic polynomials in time for $s+r=2$. Although SSPRK3 is nominally third-order, its reduction to Simpson's rule elevates its effective precision to fourth order, enabling the exact evolution of geometric moments for all $s+r \le 2$.

\begin{remark}
	Notably, the ten-stage fourth-order SSPRK method \cite{ketcheson2008highly} does not exhibit this superconvergence. Despite its higher nominal order and cost, it achieves the same geometric exactness ($s+r \le 2$) as the more efficient SSPRK3 method.
\end{remark}

\subsubsection{Proof of Theorem \ref{thm:coincide}}\label{subsec:proof_consistency}

\begin{proof}
	Lemma~\ref{lem:M_poly} establishes that the exact moment $\mathcal M_{s,r,i,j}(\tau)$ is a polynomial in $\tau$ of degree $s+r+2$. Lemma~\ref{lem:LM_exact} confirms that, for $s+r\leq2$, the semi-discrete scheme \eqref{equ:Mevo_semi} yields the exact time derivative:
	\begin{equation*}
		\mathcal L^{\widetilde{\mathcal{M}}_{s,r}}_{i,j}\bigl(I_{i,j}(\tau)\bigr)
		= \frac{\mathrm{d}\mathcal M_{s,r,i,j}(\tau)}{\mathrm{d}\tau}.
	\end{equation*}
	Furthermore, the mesh-update equation~\eqref{equ:vertex_location_RK} ensures that the SSPRK3 stages evaluate the cell geometry $I_{i,j}$ at the pseudo-time $\tau^\ell$, $\tau^\ell+\Delta \tau$, and $\tau^\ell+\frac{\Delta \tau}{2}$. Consequently, for $s+r\leq2$, we have
	\begin{equation}\label{equ:SSPRK3_stage_direv}
		\left\{
		\begin{alignedat}{2}
			\mathcal L^{\widetilde{\mathcal{M}}_{s,r}}_{i,j} \bigl(I_{i,j}^{\ell}\bigr)
			& =\mathcal L^{\widetilde{\mathcal{M}}_{s,r}}_{i,j} \bigl(I_{i,j}(\tau^\ell)\bigr)
			& & = \frac{\mathrm{d}\mathcal M_{s,r,i,j}(\tau)}{\mathrm{d}\tau} \bigg|_{\tau=\tau^\ell}, \\
			\mathcal L^{\widetilde{\mathcal{M}}_{s,r}}_{i,j} \bigl(I_{i,j}^{(1)}\bigr)
			& =\mathcal L^{\widetilde{\mathcal{M}}_{s,r}}_{i,j} \bigl(I_{i,j}(\tau^\ell+\Delta \tau)\bigr)
			& & = \frac{\mathrm{d}\mathcal M_{s,r,i,j}(\tau)}{\mathrm{d}\tau} \bigg|_{\tau=\tau^{\ell}+\Delta\tau}, \\
			\mathcal L^{\widetilde{\mathcal{M}}_{s,r}}_{i,j} \bigl(I_{i,j}^{(2)}\bigr)
			& =\mathcal L^{\widetilde{\mathcal{M}}_{s,r}}_{i,j} \bigl(I_{i,j}(\tau^\ell+\tfrac{\Delta \tau}{2})\bigr)
			& & = \frac{\mathrm{d}\mathcal M_{s,r,i,j}(\tau)}{\mathrm{d}\tau} \bigg|_{\tau=\tau^\ell+\frac{\Delta \tau}{2}}.
		\end{alignedat}
		\right.
	\end{equation}

	Finally, Lemma~\ref{lem:SSPRK3_simpson} shows that the SSPRK3 method reduces to Simpson's rule for temporal integration. This implies \textbf{fourth-order superconvergence}, as Simpson's rule is exact for cubic polynomials. Since $\partial_\tau\mathcal M_{s,r,i,j}(\tau)$ is a polynomial of degree $(s+r+2)-1 \le 3$ for $s+r \le 2$ (Lemma~\ref{lem:M_poly}), we conclude that $\widetilde{\mathcal{M}}^{\ell+1}_{s,r,i,j} = \mathcal M_{s,r,i,j}(\tau^{\ell+1})$.
\end{proof}

\subsection{Estimation of pseudo-time levels $N_\ell$ in $\mathcal{R}_h$}\label{sec:remap_Nell}

As discussed in Section \ref{subsec:o1/h}, conventional methods \cite{lipnikov2019high,lipnikov2020conservative,gu2023high} typically require $N_\ell = O(h^{-1})$ under bounded but discontinuous mesh velocity. In contrast, our remapping operator $\mathcal R_h$ satisfies the TPE(2) property independently of the mesh motion and $\Delta \tau$ (Proposition \ref{prop:ssprk3_poly_exact}). This ensures that accuracy is maintained regardless of temporal resolution, leaving $\Delta \tau$ constrained solely by the stability requirement:
\begin{equation}\label{equ:cfl_practical}
	\Delta \tau \;\leq\;
	C_{\mathrm{CFL}}\,
	\min_{i,j}\frac{{\mathcal M}_{0,0,i,j}}
	{\max_{k=1}^{4} \left( a^{\max}_{i,j,k}\,|l_{i,j}^k| \right)}.
\end{equation}
where $a^{\max}_{i,j,k}=\max_{g}\bigl|\mathbf w_{k,g}(I_{i,j})\!\cdot\!\mathbf n_k(I_{i,j})\bigr|$ and $C_{\mathrm{CFL}}>0$ is the CFL number. 

To estimate $N_\ell$, we assume uniform pseudo-time levels over the interval $[0,\tau^{\mathrm{final}}]$ with $\Delta\tau = \tau^{\mathrm{final}}/N_\ell$. Let $\delta\mathbf x_P=\mathbf x^{n+1}_P-\mathbf x^n_P$ for $P \in \{A, B, C, D\}$ denote the displacements of the four vertices of cell $I_{i,j}=ABCD$ during remapping. Since the mesh velocity along each edge is determined via linear interpolation of the endpoint velocities, its maximum magnitude is attained at the vertices. Combining this property with the mesh-velocity definition in~\eqref{equ:grid_vel}, we obtain the bound $|\mathbf{w}| \leq \max_P |\delta\mathbf{x}_{P}|/\tau^{\mathrm{final}}$. Substituting this inequality and the definition of $a^{\max}_{i,j,k}$ into the CFL condition~\eqref{equ:cfl_practical} yields the following sufficient condition for $N_\ell$:
\begin{equation}\label{equ:N_bound}
	N_\ell \;\geq\;
	\max_{i,j}\frac{\max_{P}|\delta \mathbf x_{P}|\,\max_{k=1}^{4}|l_{i,j}^k|}
	{C_{\mathrm{CFL}}\,{\mathcal M}_{0,0,i,j}}.
\end{equation}
Given standard geometric scalings where $\Delta t = \mathcal{O}(h)$, $|l_{i,j}^k| = \mathcal{O}(h)$, and ${\mathcal M}_{0,0,i,j} = \mathcal{O}(h^2)$, and assuming a bounded physical-time mesh velocity satisfying $|\mathbf{w}^{\text{phys}}|=|\delta \mathbf{x}_{P}/\Delta t| \lesssim 1$, the estimate \eqref{equ:N_bound} implies $N_\ell = \mathcal{O}(1)$. Consequently, even in the presence of discontinuous mesh velocity, only a \emph{constant number of pseudo-time levels} is required to maintain both stability and high-order accuracy, thereby reducing the remapping levels per physical step from $\mathcal{O}(h^{-1})$ to $\mathcal{O}(1)$. This claim is further substantiated by the numerical tests in Section~\ref{sec:numerical_tests}.

\begin{remark}\label{remark:tautarget_arbitrary}
	In the semi-discrete schemes \eqref{equ:Mevo_semi}--\eqref{equ:V_semi}, the numerical updates depend solely on the product $\mathbf{w}\Delta \tau$. Since the CFL condition \eqref{equ:cfl_practical} scales $\Delta \tau$ inversely proportional to $\mathbf{w}$, this product remains invariant with respect to the scaling parameter $\tau^{\mathrm{final}}$, rendering the numerical solution independent of the choice of $\tau^{\mathrm{final}}$. We therefore set $\tau^\mathrm{final} = \Delta t$ for simplicity.
\end{remark}

\begin{remark}
	Satisfying \eqref{equ:cfl_practical} with $C_{\mathrm{CFL}}\!\le\!\frac14$ implies the positivity of the evolved cell volume $\widetilde{\mathcal{M}}_{0,0,i,j}$. This follows directly from the semi-discrete scheme \eqref{equ:Mevo_semi}.
\end{remark}

\section{Numerical experiments}\label{sec:numerical_tests}
This section evaluates the performance of the proposed TPE(2) RMM scheme (hereafter the \textbf{TPE(2) scheme}) via carefully designed numerical experiments. The primary objectives are to verify: (i) the TPE(2) property; (ii) accuracy and efficiency under discontinuous mesh velocity; and (iii) the non-oscillatory shock-capturing capability. All experiments are implemented in C++ using double-precision arithmetic and are executed on a Linux server equipped with an Intel\textsuperscript{\textregistered} Xeon\textsuperscript{\textregistered} Gold 6348 CPU at 2.60\,GHz.

To isolate the benefits of the TPE(2) scheme, we compare it with two closely related variants that differ \emph{only} in the geometric treatment during the remapping step:
\begin{enumerate}
	\item \textbf{GCL scheme.} In the WENO reconstruction, we use the evolved cell volume $\widetilde{\mathcal M}_{0,0,i,j}$ and compute the exact integrals of the basis functions $\psi_{s,r}(x,y)$ (defined in \eqref{eq:norm_basis}) for $s+r>0$. This ensures the classical GCL.
	
	\item \textbf{NonGCL scheme.} All geometric moments $\mathcal{M}_{s,r}$ are evaluated exactly through \eqref{equ:shape_discre_int}, in alignment with \cite{gu2023high}.
\end{enumerate} 

Furthermore, to assess robustness under discontinuous mesh velocity, we compare two pseudo-time stepping strategies:

%

\begin{enumerate}
	\item \textbf{Resolution-dependent levels (rl).} 
	This strategy calculates $\tau^{\mathrm{final}}$ and $\Delta \tau$ following \cite{gu2023high}, thereby enforcing Lipschitz continuity of the pseudo-time mesh velocity. When the underlying physical-time mesh velocity is bounded yet discontinuous, the required number of pseudo-time levels is $N_\ell = \mathcal{O}(h^{-1})$.
	
	\item \textbf{Order-one levels (ol).} 
	By setting $\tau^{\mathrm{final}}=\Delta t$ and determining $\Delta \tau$ via the CFL constraint~\eqref{equ:cfl_practical}, this strategy only requires $N_\ell = \mathcal{O}(1)$ for bounded and discontinuous physical-time mesh velocity.
\end{enumerate}

Unless stated otherwise, simulations are initialized on a uniform $N_x\times N_y$ Cartesian grid with spacing $h_x\times h_y$. We employ two rezoning strategies:
\begin{enumerate}
	\item \textbf{Random rezoning.} Vertices are perturbed from their initial uniform positions via:
	\[
	\mathbf{x}^n = \mathbf{x}^0 + C_r\,(U_1 h_x,\; U_2 h_y)^\top + \mathbf b t,
	\]
	where $\mathbf b \in \mathbb R^2$ is a constant vector and $U_1$, $U_2$ are independent random variables uniformly distributed on $(-0.5, 0.5)$, yielding a bounded yet discontinuous physical-time mesh velocity. A fixed random seed (\texttt{std::mt19937 gen(123456789u)}) is used to ensure reproducibility. The mesh satisfies the regularity condition in Assumption~\ref{assump:non-twisting} provided $|C_r| \leq 0.5$.
	
	\item \textbf{ALE rezoning.} The mesh is evolved by combining a Lagrangian step with a subsequent mesh optimization procedure~\cite{cheng5472825fifth}.
\end{enumerate}

\begin{figure}[!tbhp]
	\centering
	\captionsetup[subfigure]{labelformat=empty}
	\includegraphics[width=0.95\textwidth]{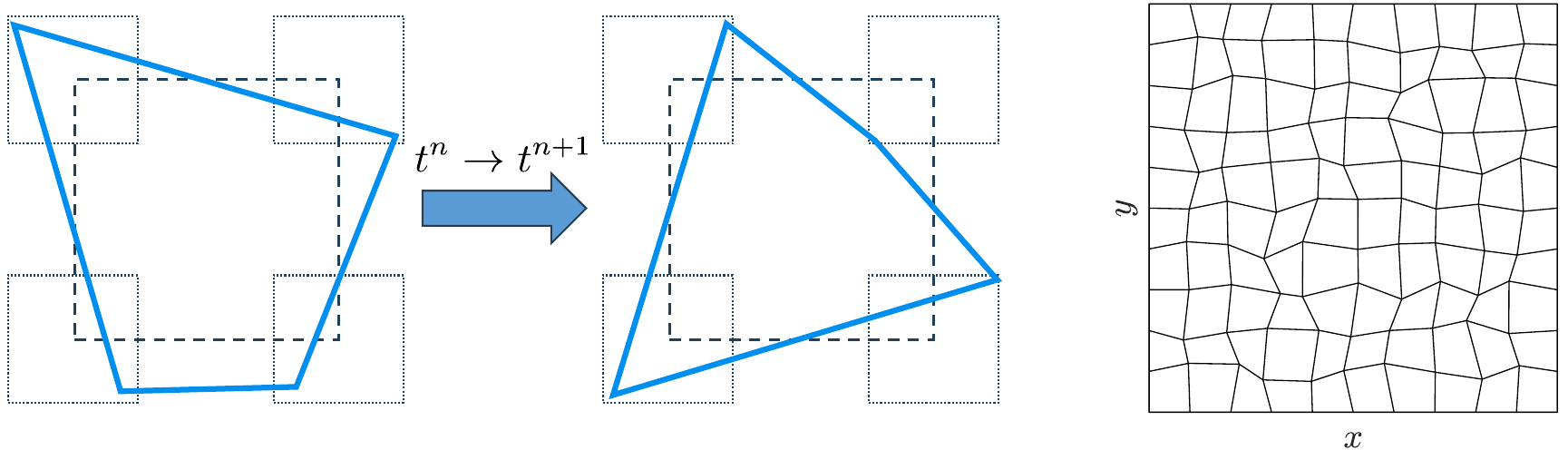}
	\captionsetup{font=small}
	\caption{Mesh configuration for the random rezoning with $|C_r|=0.5$. The mesh undergoes large-scale random perturbations at every physical-time step, resulting in severe distortion and occasional near-triangular cells.}
	\label{fig:distor}
\end{figure}

\subsection{Transport polynomial exactness tests}\label{subsec:tpe_tests}

To verify transport polynomial exactness, we solve the linear advection equation $u_t + u_x + u_y = 0$ on $\Omega=[0,1]^2$ using four schemes: TPE(2)-ol, GCL-ol, NonGCL-ol, and NonGCL-rl. Simulations are conducted on a $40\times40$ mesh, with exact boundary conditions imposed at each stage of the SSPRK3 method \cite{Carpenter1995}. We employ a random rezoning strategy with $C_r=0.5$ and $\mathbf b=(-0.6,-0.8)^{\top}$, which induces discontinuous mesh velocities and severe mesh deformations, as shown in Figure \ref{fig:distor}.

To ensure generality, we test three sets of randomized initial conditions:

\subsubsection{Quadratic initial condition (TPE(2)) tests}\label{subsubsec:tpe2_test}
$u_0(x,y)=\hat U_1+\hat U_2 x+\hat U_3 y+\hat U_4 x^2+\hat U_5 xy+\hat U_6 y^2.$

\subsubsection{Linear initial condition (TPE(1)) tests}\label{subsubsec:tpe1_test}
$u_0(x,y)=\hat U_1+\hat U_2 x+\hat U_3 y.$

\subsubsection{Free-stream preservation (TPE(0)) tests}\label{subsubsec:tpe0_test}
$u_0(x,y)=\hat U_1.$

The coefficients $\hat U_1, \dots, \hat U_6$ are independent random variables uniformly distributed in $(-5, 5)$, generated using a fixed seed (\texttt{std::mt19937 gen(202688888888u)}). This ensures that all schemes are evaluated on identical problem instances.

Table~\ref{tab:t1-1} summarizes the maximum $L^1$ and $L^\infty$ errors at $T=0.1$ and the average pseudo-time levels over the 20 runs for each test. Notably, the TPE(2)-ol scheme achieves machine precision across all $3\times20$ tests. In contrast, the GCL-ol scheme maintains exactness only for the free-stream tests, exhibiting noticeable errors in the quadratic and linear tests. The NonGCL schemes incur substantially larger deviations in all tests. These results confirm that while the classical GCL suffices for free-stream preservation, it is insufficient to guarantee TPE($k$) for $k \ge 1$, underscoring the critical role of EGMs in achieving the general TPE($k$) property.

\begin{table}[!t]
	\centering
	\captionsetup{font=small}
	\caption{Transport polynomial exactness tests \ref{subsubsec:tpe2_test}--\ref{subsubsec:tpe0_test}. Maximum errors and average pseudo-time levels $\overline{N_\ell}$ for solving \(u_t+u_x+u_y=0\) with $3\times20$ randomized polynomial initial values. \(T=0.1\), \(N_x\times N_y=40\times40\).}
	
	\begingroup
	\setlength{\tabcolsep}{2.6666pt} 
	\renewcommand{\arraystretch}{1.2} 
	\centering
	\footnotesize
	\label{tab:t1-1}
	\begin{tabular}{ccccccc} 
		\bottomrule[1.0pt]
		&& & TPE(2)-ol & GCL-ol & NonGCL-ol & NonGCL-rl\\
		\hline
		&\multirow{3}{*}{\shortstack{Quadratic tests\\($u_0 \in \mathbb{P}_2^1$)}}&$L^{1}$ error&7.13E-15&1.88E-04&1.37E-02&1.11E-02\\
		&&$L^{\infty}$ error&4.80E-14&3.73E-03&1.48E-01&1.21E-01\\
		&&$\overline{N_\ell}$&1.9&1.9&1.9&2.1\\
		\hline
		&\multirow{3}{*}{\shortstack{Linear tests\\($u_0 \in \mathbb{P}_1^1$)}}&$L^{1}$ error&6.30E-15&1.06E-04&1.10E-02&8.82E-03\\
		&&$L^{\infty}$ error&2.40E-14&1.08E-03&6.30E-02&6.21E-02\\
		&&$\overline{N_\ell}$&1.9&1.9&1.9&2.1\\
		\hline
		&\multirow{3}{*}{\shortstack{Free-stream tests\\($u_0 \in \mathbb{P}_0^1$)}}&$L^{1}$ error&4.93E-15&5.00E-15&5.85E-03&4.80E-03\\
		&&$L^{\infty}$ error&1.20E-14&1.11E-14&6.54E-02&4.78E-02\\
		&&$\overline{N_\ell}$&1.9&1.9&1.9&2.1\\
		\toprule[1.0pt]
	\end{tabular}
	\endgroup
\end{table}

\subsection{Accuracy test for the compressible Euler equations}\label{subsec:euler_random_sine}
We assess the accuracy and efficiency of the TPE(2)-ol scheme using the 2D compressible Euler equations \cite{toro2013riemann} with smooth thermodynamic profiles under \textbf{discontinuous mesh velocity}. 
The vector of conserved variables $\mathbf U$ is uniquely determined by the primitive variables $(\rho, v_x, v_y, p)$ and the adiabatic index $\gamma$. In this problem, the initial condition is given by
\[
(\rho, v_x, v_y, p, \gamma) = \bigl(1+0.2\sin(2\pi(x+y)),\; 1,\; 1,\; 1,\;1.4\bigr).
\]
The computational domain is defined as $\Omega=[0,1]^2$. Periodic boundary conditions are imposed on $\partial\Omega$, and the simulation is conducted up to $T=0.1$. The interior mesh is rezoned using the random strategy with $C_r=0.5$ and $\mathbf{b}=(-0.6,-0.8)$, while boundary vertices move with the constant velocity $\mathbf{b}$.

Table~\ref{tab:t2-1} summarizes the convergence errors, average pseudo-time levels $\overline{N_\ell}$, and the maximum mesh-velocity Lipschitz constant $L_{\mathbf{w}}^{\max}$. Here, $L_{\mathbf{w}}^{\max}$ is calculated as the global maximum $L^\infty$-norm of $\nabla\mathbf{w}$, evaluated using the discretization in~\cite[Section 2.2]{gu2023high}. The results show that under the ``ol'' strategy, $L_{\mathbf{w}}^{\max}$ scales as $\mathcal{O}(h^{-1})$, consistent with the presence of discontinuous mesh velocity, whereas the ``rl'' strategy~\cite{gu2023high} effectively bounds the Lipschitz constant. 
Remarkably, despite the unbounded growth of $L_{\mathbf{w}}^{\max}$ in the ``ol'' case, the TPE(2)-ol scheme achieves the lowest error magnitudes and maintains stable third-order convergence. 
Conversely, the GCL-ol scheme exhibits severe order reduction, and the NonGCL-ol scheme fails to converge. 
Although the ``rl'' strategy enables the NonGCL-rl scheme to recover third-order accuracy, it remains less accurate than the TPE(2)-ol scheme and incurs a substantial computational penalty. 
Specifically, on the $320\times320$ mesh, it requires $\overline{N_\ell}=19.7$---nearly a tenfold increase compared to the ``ol'' strategy---rendering its remapping step an order of magnitude slower than that of the TPE(2)-ol scheme.

Table~\ref{tab:t2-2} further reveals that even with the same ``ol'' strategy, the TPE(2)-ol and GCL-ol schemes are nearly three times faster than the NonGCL-ol scheme. This efficiency gap stems primarily from the geometric mismatches in the NonGCL scheme, which induce spurious oscillations. These oscillations trigger the KXRCF indicator to erroneously flag smooth regions as discontinuities, forcing the hybrid WENO reconstruction to employ the expensive non-oscillatory stencils. By ensuring geometric compatibility, the use of EGMs prevents such false positives, thereby preserving the intrinsic efficiency of the hybrid strategy \cite{li2010hybrid}. 

\begin{table}[!t]
	\centering
	\captionsetup{font=small}
	\caption{Random-mesh accuracy test \ref{subsec:euler_random_sine}. Errors of $\rho$, average pseudo-time levels $\overline{N_\ell}$, and maximum mesh-velocity Lipschitz constants $L_{\mathbf{w}}^{\max}$ of different schemes. Final time $T=0.1$.}
	
	\begingroup
	\setlength{\tabcolsep}{2.6666pt} 
	\renewcommand{\arraystretch}{1.2} 
	\centering
	\footnotesize
	\label{tab:t2-1}
	\begin{tabular}{ccccccccc} 
		\bottomrule[1.0pt]
		& $N_{x}\times N_{y}$& $80\times 80$ & $120\times 120$& $160\times 160$& $200\times 200$&$240\times 240$ & $280\times 280$ & $320\times 320$\\
		\hline
		\multirow{6}{*}{TPE(2)-ol}&$L^{1}$ error&1.11E-04&3.35E-05&1.42E-05&7.35E-06&4.26E-06&2.70E-06&1.82E-06\\
		&Order&---&2.93&2.95&2.97&2.98&2.95&2.97\\
		&$L^{\infty}$ error&1.84E-04&5.69E-05&2.43E-05&1.25E-05&7.68E-06&5.37E-06&3.27E-06\\
		&Order&---&2.90&2.95&2.99&2.67&2.32&3.73\\
		&$\overline{N_\ell}$&2.0&2.0&2.0&2.0&2.0&2.0&2.0\\
		&$L_{\mathbf{w}}^{\max}$&1874.9&2849.1&3900.7&4998.5&5991.7&7318.3&8443.5\\
		\hline
		\multirow{6}{*}{GCL-ol}&$L^{1}$ error&1.12E-04&3.45E-05&1.53E-05&8.51E-06&5.66E-06&3.78E-06&2.79E-06\\
		&Order&---&2.91&2.83&2.63&2.24&2.62&2.25\\
		&$L^{\infty}$ error&2.02E-04&6.56E-05&4.83E-05&4.84E-05&5.35E-05&3.71E-05&3.04E-05\\
		&Order&---&2.78&1.06&-0.00&-0.55&2.37&1.49\\
		&$\overline{N_\ell}$&2.0&2.0&2.0&2.0&2.0&2.0&2.0\\
		&$L_{\mathbf{w}}^{\max}$&1874.9&2849.1&3900.7&4998.5&5991.7&7318.3&8443.5\\
		\hline
		\multirow{6}{*}{NonGCL-ol}&$L^{1}$ error&1.72E-03&1.72E-03&1.59E-03&1.59E-03&1.62E-03&1.53E-03&1.51E-03\\
		&Order&---&0.00&0.26&0.02&-0.11&0.36&0.13\\
		&$L^{\infty}$ error&1.83E-02&1.68E-02&1.68E-02&2.24E-02&1.86E-02&2.05E-02&1.94E-02\\
		&Order&---&0.21&-0.01&-1.28&1.02&-0.64&0.44\\
		&$\overline{N_\ell}$&2.0&2.0&2.0&2.0&2.0&2.0&2.0\\
		&$L_{\mathbf{w}}^{\max}$&1873.8&2846.9&3900.6&5080.2&5991.8&7452.6&8444.4\\
		\hline
		\multirow{6}{*}{NonGCL-rl}&$L^{1}$ error&1.53E-04&3.89E-05&1.55E-05&7.93E-06&4.47E-06&2.85E-06&1.93E-06\\
		&Order&---&3.36&3.20&3.00&3.14&2.92&2.91\\
		&$L^{\infty}$ error&1.05E-03&5.79E-04&1.90E-04&7.94E-05&4.35E-05&1.91E-05&2.10E-05\\
		&Order&---&1.47&3.86&3.92&3.30&5.33&-0.72\\
		&$\overline{N_\ell}$&4.7&7.0&9.4&12.0&14.4&17.1&19.7\\
		&$L_{\mathbf{w}}^{\max}$&10.0&10.0&10.0&10.0&10.0&10.0&10.0\\
		\toprule[1.0pt]
	\end{tabular}
	\endgroup
\end{table}

\begin{table}[!t]
	\centering
	\captionsetup{font=small}
	\caption{Random-mesh test \ref{subsec:euler_random_sine} on a $320\times320$ mesh. Total CPU time, remapping time, remapping fraction, and average number of pseudo-time levels $\overline{N_\ell}$ at $T=0.1$.}
	
	\begingroup
	\setlength{\tabcolsep}{2.6666pt} 
	\renewcommand{\arraystretch}{1.2} 
	\centering
	\footnotesize
	\label{tab:t2-2}
	\begin{tabular}{cccccc} 
		\bottomrule[1.0pt]
		& & TPE(2)-ol & GCL-ol & NonGCL-ol & NonGCL-rl\\
		\hline
		&Total CPU time (s)&5795.1&6211.8&16490.4&67168.3\\
		&Remapping time (s)&3942.6&4365.8&10975.8&64781.5\\
		&Remapping fraction&68.0\%&70.3\%&66.6\%&96.4\%\\
		&$\overline{N_\ell}$&2.0&2.0&2.0&19.7\\
		\toprule[1.0pt]
	\end{tabular}
	\endgroup
\end{table}

\subsection{Assessment of ALE rezoning efficacy}\label{subsec:ale_rezoning_tests}
To verify that the incorporation of EGMs does not compromise the resolution gains of mesh adaptation, we evaluate the TPE(2)-ol scheme using quasi-1D Euler benchmarks. Specifically, we compare the scheme with ALE rezoning against a fixed-mesh baseline (setting $\mathbf w=\mathbf 0$ in the TPE(2)-ol scheme, denoted as the \textit{Fixed} scheme).

\subsubsection{Woodward--Colella blast wave}\label{subsubsec:blast_wave}
The interacting blast wave problem is employed to assess the schemes' robustness and their ability to resolve complex patterns of interacting strong shocks, rarefactions, and contact discontinuities. The initial condition is 
\[
(\rho, v_x, v_y, p, \gamma)=
\begin{cases}
	(1, 0, 0, 10^{3}, 1.4), & x < 0.1,\\
	(1, 0, 0, 10^{-2}, 1.4), & 0.1 \le x < 0.9,\\
	(1, 0, 0, 10^{2}, 1.4), & x \ge 0.9.
\end{cases}
\]
Computations are performed on a $400 \times 10$ mesh over $[0,1] \times [-5h_x, 5h_x]$ with reflective boundaries at $x=0$ and $x=1$. 
Figure~\ref{fig:blast} displays the density profile along $y=0$ at $T=0.038$. Both schemes successfully capture the strong discontinuities without generating spurious oscillations. However, the zoom-in views reveal that the TPE(2)-ol scheme, by virtue of the ALE-rezoning strategy, achieves a significantly higher resolution of intricate structures compared to the Fixed scheme.

\begin{figure}[!htb]
	\centering
	\subfloat[Overview.]{\includegraphics[height=0.24\textheight]{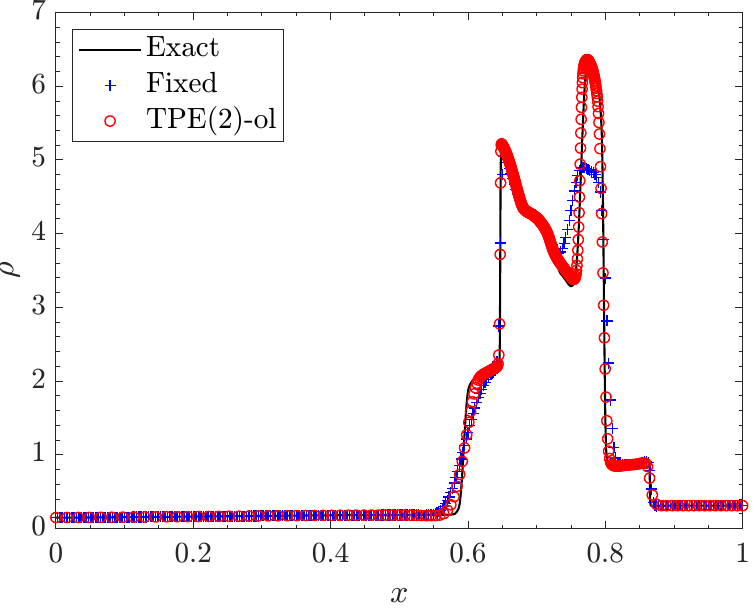}}
	\subfloat[Zoom-in.]{\includegraphics[height=0.24\textheight]{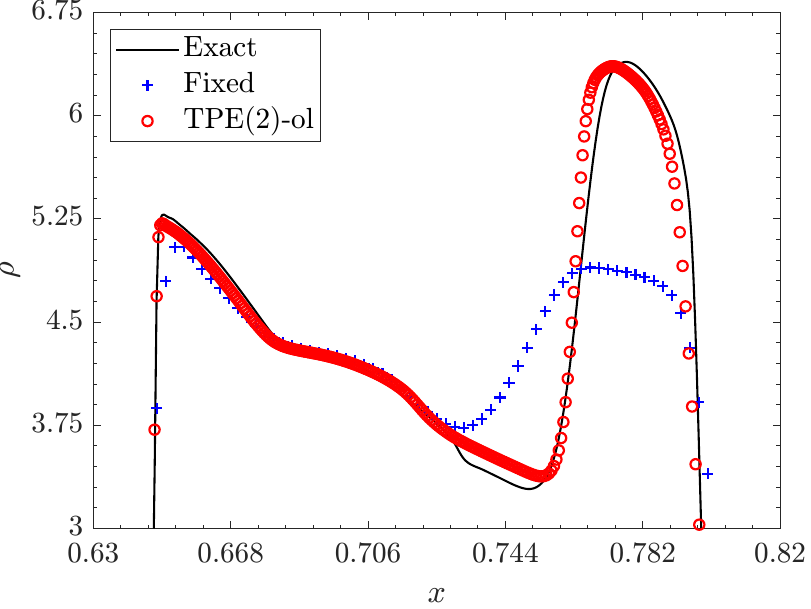}}
	\captionsetup{font=small}
	\caption{Woodward--Colella blast wave problem \ref{subsubsec:blast_wave}: density profile along $y=0$ computed by the TPE(2)-ol and Fixed schemes. $N_x\times N_y=400\times 10$, $T=0.038$.}
	\label{fig:blast}
\end{figure}

\subsubsection{Shu--Osher problem}\label{subsubsec:shu_osher}
We next simulate the classical Shu--Osher problem. The initial condition is given by
\[
(\rho, v_x, v_y, p, \gamma)=
\begin{cases}
	(3.857143,\ 2.629369,\ 0,\ 10.33333,\ 1.4), & x<-4,\\
	(1+0.2\sin(5x),\ 0,\ 0,\ 1,\ 1.4), & x\ge -4.
\end{cases}
\]
Computations are performed on a $400 \times 10$ grid.
The Fixed scheme is solved on $[-5,5]\times[-5h_x,5h_x]$.
In contrast, to accommodate the Lagrangian mesh compression inherent to the ALE strategy, the TPE(2)-ol computation is initialized on an extended domain $[-9.5,5]\times[-5h_x,5h_x]$.
Consequently, the ALE simulation begins with a \textit{coarser} initial grid spacing than the fixed-mesh baseline.
Despite this, Figure~\ref{fig:shu} reveals that the ALE rezoning strategy achieves substantially better resolution of the post-shock fine-scale structures, confirming the benefit of grid adaptation.

\begin{figure}[!htb]
	\centering
	\subfloat[Overview.]{\includegraphics[height=0.24\textheight]{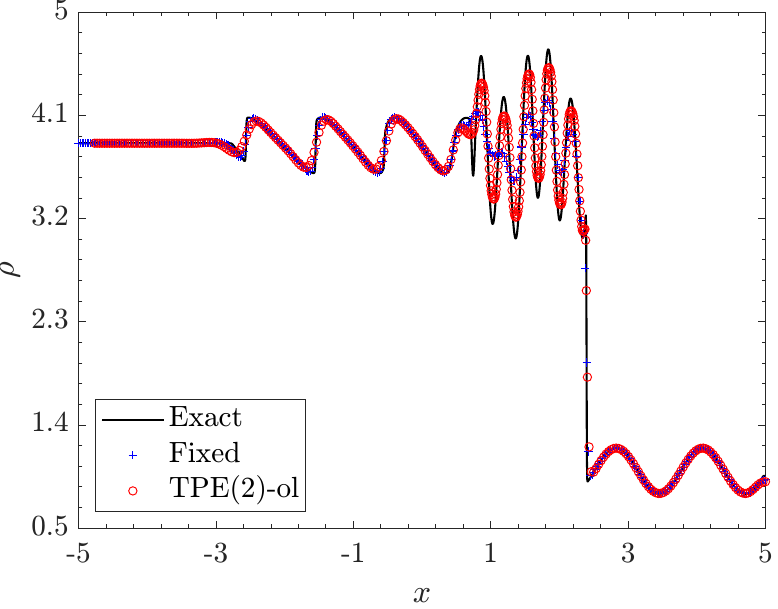}}
	\subfloat[Zoom-in.]{\includegraphics[height=0.24\textheight]{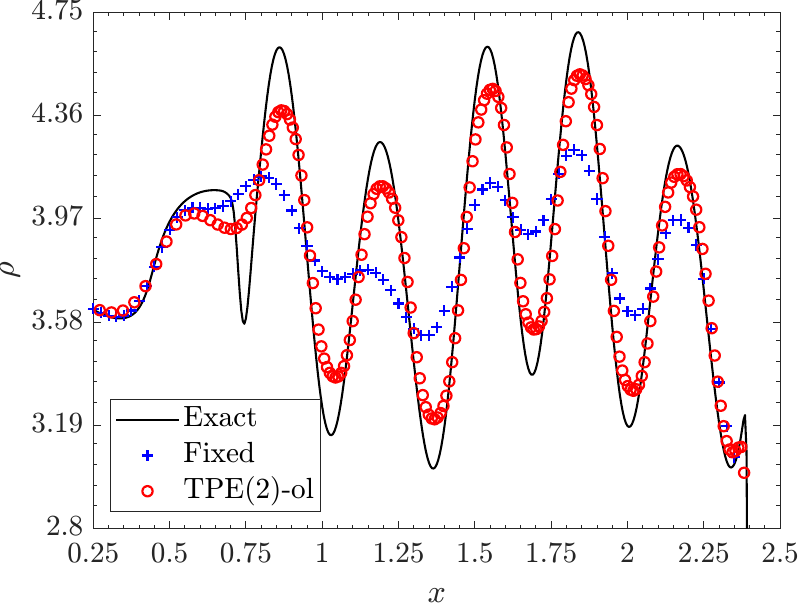}}
	\captionsetup{font=small}
	\caption{Shu--Osher problem \ref{subsubsec:shu_osher}: density profile along $y=0$ computed by the TPE(2)-ol and Fixed schemes. $N_x\times N_y=400\times 10$, $T=1.8$.}
	\label{fig:shu}
\end{figure}

\subsection{Discontinuous problems on large-deformation random meshes}
\label{subsec:discont_random}
This section assesses the stability of the TPE(2)-ol scheme for discontinuous Euler-equation benchmarks under random rezoning. We set the rezoning parameters to \(C_r=0.5\) and \(\mathbf{b}=\mathbf{0}\), which, as illustrated in Figure~\ref{fig:distor}, induces extreme mesh distortion and skewness. This provides a stringent stress test for the schemes' robustness against large, rapid deformations.

\subsubsection{Sensitivity to pseudo-time levels $N_\ell$ in a 2D Riemann problem}\label{subsubsec:2D_Riemann}
We employ a 2D Riemann problem to assess the sensitivity of the TPE(2) and NonGCL schemes to the number of remapping pseudo-time levels $N_\ell$. 
The initial condition, which involves two stationary contact discontinuities and two shocks, is given by
\[
(\rho, v_x, v_y, p, \gamma)=
\begin{cases}
	(0.8,\ 0,\ 0,\ 1,\ 1.4), & x<0.5,\ y<0.5,\\
	(1,\ 0.7276,\ 0,\ 1,\ 1.4), & x<0.5,\ y\ge 0.5,\\
	(1,\ 0,\ 0.7276,\ 1,\ 1.4), & x\ge 0.5,\ y<0.5,\\
	(0.5313,\ 0,\ 0,\ 0.4,\ 1.4), & x\ge 0.5,\ y\ge 0.5.
\end{cases}
\]
The computational domain is $\Omega=[0,1]^2$, with outflow boundary conditions prescribed on all boundaries.

Figure~\ref{fig:2DRiemann1} compares the density contours computed with varying $N_\ell$.
The NonGCL scheme exhibits a strong dependence on $N_\ell$: significant distortion and loss of detail are evident at $N_\ell=2$, which are mitigated only as $N_\ell$ is increased to 6.
In sharp contrast, the TPE(2) scheme demonstrates superior robustness. Its solution at $N_\ell=2$ is already comparable to the NonGCL result at $N_\ell=6$, with negligible variation for higher $N_\ell$.
This confirms the efficiency advantage of the TPE(2) scheme: accurate remapping is achieved with minimal pseudo-time levels, substantially reducing the computational cost.

\begin{figure}[!htb]
	\centering
	\subfloat[TPE(2) (\(N_\ell=2\)).]{\includegraphics[height=0.19\textheight]{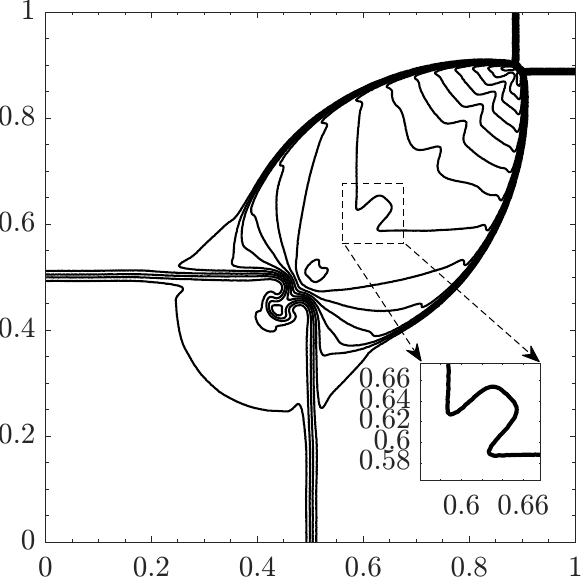}}
	\subfloat[TPE(2) (\(N_\ell=3\)).]{\includegraphics[height=0.19\textheight]{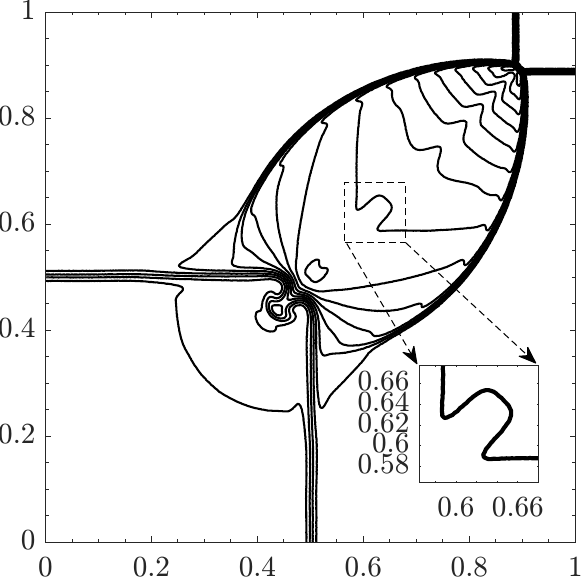}}
	\subfloat[TPE(2) (\(N_\ell=6\)).]{\includegraphics[height=0.19\textheight]{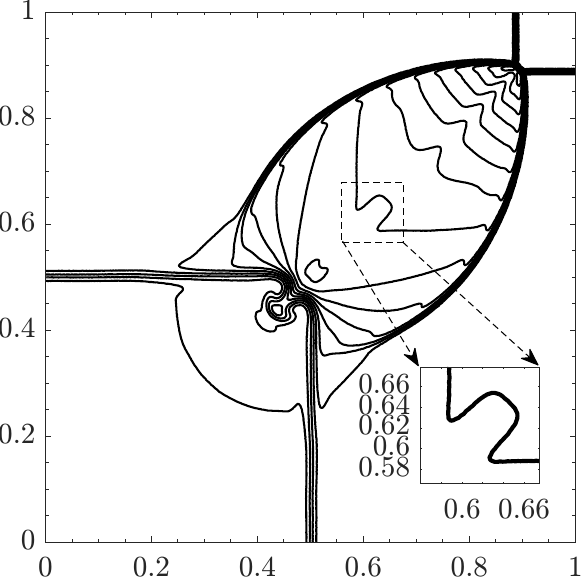}}
	
	\subfloat[NonGCL (\(N_\ell=2\)).]{\includegraphics[height=0.19\textheight]{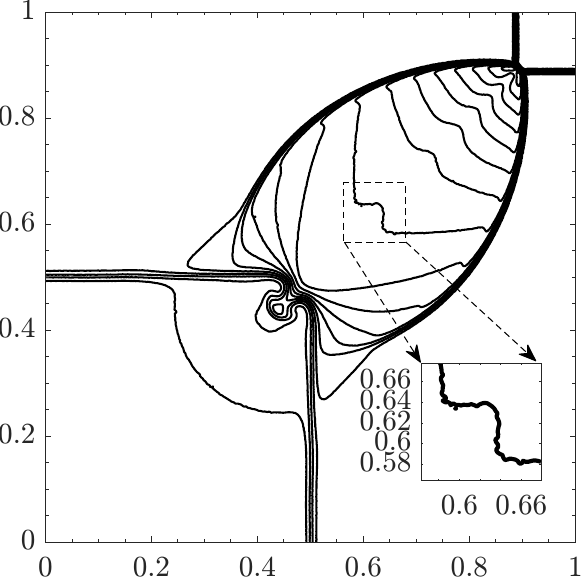}}
	\subfloat[NonGCL (\(N_\ell=3\)).]{\includegraphics[height=0.19\textheight]{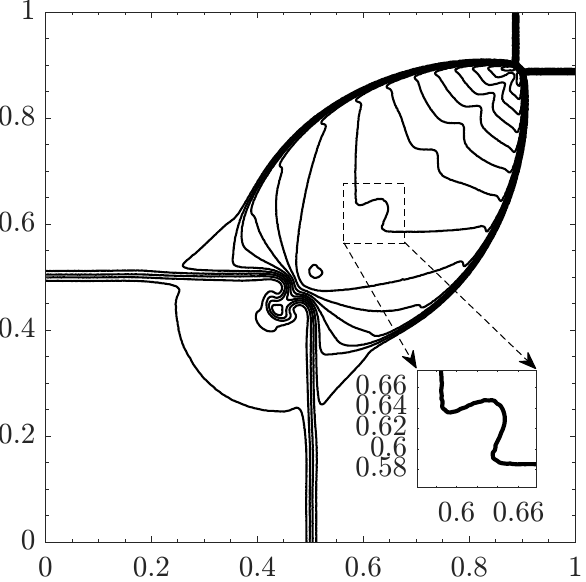}}
	\subfloat[NonGCL (\(N_\ell=6\)).]{\includegraphics[height=0.19\textheight]{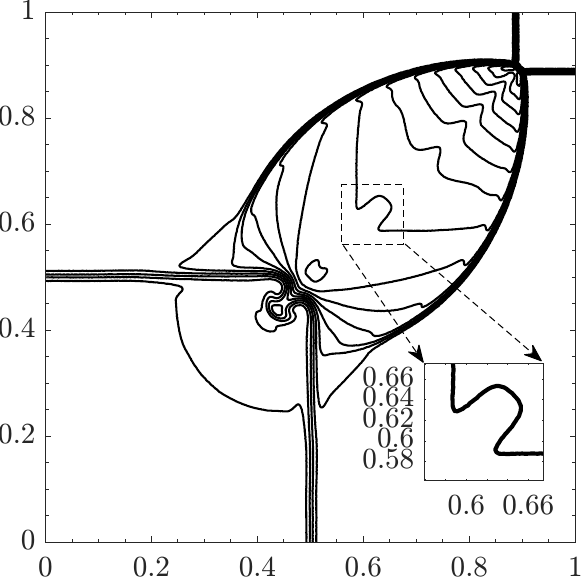}}
	
	\captionsetup{font=small}
	\caption{2D Riemann problem \ref{subsubsec:2D_Riemann} under extreme random rezoning ($C_r=0.5$, $\mathbf{b}=\mathbf{0}$):
	density contours computed by the TPE(2) and NonGCL schemes with different numbers of pseudo-time levels \(N_\ell\).
	The mesh is randomly perturbed at every time step to induce severe distortion. 
	Each panel uses 23 contour levels.
	\(N_x\times N_y=400\times 400\), \(T=0.25\).}
	\label{fig:2DRiemann1}
\end{figure}

\subsubsection{Double Mach reflection under extreme mesh distortion}\label{subsubsec:double_mach}
Finally, we consider the double Mach reflection problem, a standard benchmark for evaluating a scheme's ability to resolve fine-scale structures involving strong shock interactions. 
A Mach 10 shock impinges on a reflecting wall at a \(60^\circ\) angle in the domain \(\Omega=[0,4]\times[0,1]\). With the standard setup \cite{qiu2005runge,peng2025oedg}, the solution is advanced to $T=0.2$. Figure~\ref{fig:doublemach} displays the density contours, showing excellent agreement with the numerical results reported in \cite{qiu2005runge,peng2025oedg}.
Despite the severe mesh deformation, the TPE(2)-ol scheme resolves intricate flow features without mesh-induced artifacts or spurious oscillations, providing strong evidence of its robustness.

\begin{figure}[!htb]
	\centering
	\subfloat{\includegraphics[height=0.15\textheight]{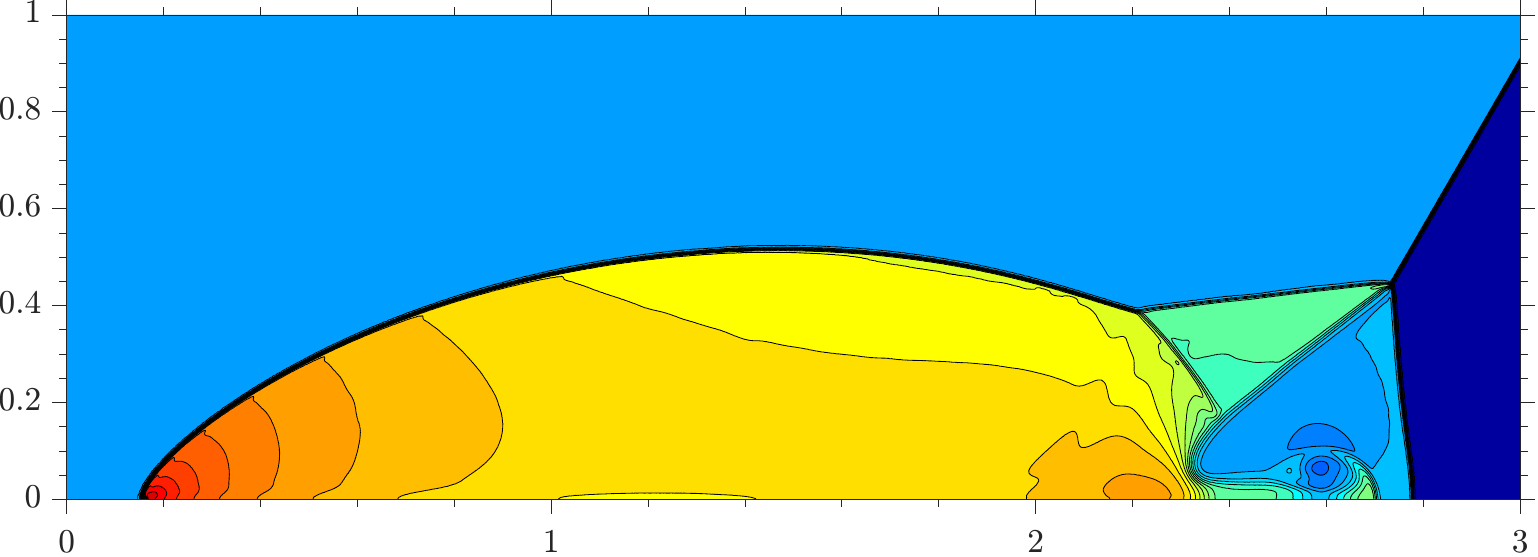}}
	\subfloat{\includegraphics[height=0.15\textheight]{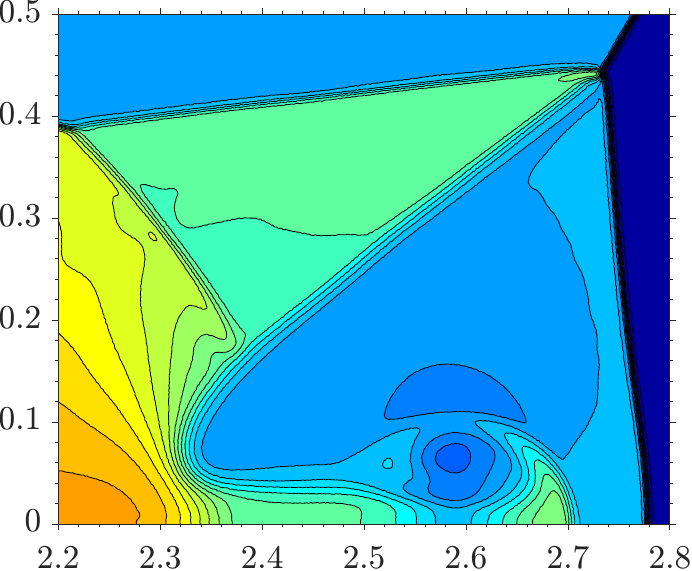}}
	\captionsetup{font=small}
	\caption{Double Mach reflection problem \ref{subsubsec:double_mach} under extreme random rezoning ($C_r=0.5$, $\mathbf{b}=\mathbf{0}$):
		density contours computed by the TPE(2)-ol scheme.
		The mesh undergoes large-scale random perturbations at every time step.
		\(N_x\times N_y=1920\times 480\), \(T=0.2\).}
	\label{fig:doublemach}
\end{figure}
\section{Conclusions}\label{sec:conclusion}
This work addresses the challenge of maintaining formal high-order accuracy in moving-mesh methods under nonsmooth mesh motion. By establishing the \emph{transport polynomial exactness} (TPE($k$)) property, we provide a rigorous criterion that generalizes classical free-stream preservation. This criterion shifts the focus from simple volume consistency to high-order fidelity under arbitrary mesh motion. To realize the TPE($k$) property, we propose the method of \emph{evolved geometric moments} (EGMs). By solving auxiliary transport equations discretized compatibly with the conserved variables, this approach \emph{extends the classical geometric conservation law} (GCL) to higher orders and can be integrated into arbitrary Lagrangian--Eulerian (ALE) frameworks. Analytically, we prove that the alignment between second-degree EGMs and exact geometric moments is intrinsic to the \emph{superconvergence behavior} of the third-order strong stability preserving Runge--Kutta method. This mathematical synergy guarantees geometric consistency without resorting to nonconservative post-processing. Leveraging this theoretical foundation, we construct a conservative two-dimensional finite-volume TPE(2) rezoning moving-mesh (RMM) scheme. By incorporating EGMs with a 2-exact hybrid WENO reconstruction, the scheme ensures that both the remapping operator and the physical evolution operator satisfy the TPE(2) property. Crucially, this property holds \emph{independently of mesh-velocity regularity} and the \emph{physical and pseudo-time step sizes}.  The decoupling of accuracy from mesh-velocity regularity allows the solver to operate with $\mathcal{O}(1)$ pseudo-time levels when handling bounded yet discontinuous mesh velocity, breaking the efficiency bottleneck of conventional advection-based remapping. In the random-mesh Euler test, the number of pseudo-time levels $N_\ell$ remains $2$ up to a $320\times 320$ mesh, whereas the resolution-dependent strategy following \cite{gu2023high} reaches $N_\ell\approx 20$ and results in roughly a tenfold increase in remapping time. Complementing this performance leap, the EGM method streamlines implementation by avoiding repeated reference-to-physical mappings and Jacobian evaluations. 

Fundamentally, our results reveal a critical insight: while the GCL ensures free-stream preservation (TPE(0)), it is \emph{insufficient} to prevent order degradation in high-order methods under extreme conditions, validating the necessity of the TPE($k$) framework. This work challenges the prevailing assumption that high-order accuracy in ALE methods necessitates mesh-velocity smoothness. By embedding geometric information directly into the conservation laws via EGMs, we demonstrate that robustness can be maintained even under discontinuous mesh velocity, paving the way for next-generation high-order algorithms robust to extreme mesh motion.


\appendix
\bibliographystyle{siamplain}
\bibliography{ref}
\newpage 
\end{document}